%% file: main.tex
\documentclass{amsart}
\usepackage{preamble}

\begin{document}

\title{
Quaternionic Nevanlinna Functions}
 \author[M. Ammar]
{Muhammad Ammar*}
\date{\today}

\thanks{*This work was carried out within the framework of the SIRS (Scientific Inquiry and Research Students) program at Niles Township High School West. All work presented herein is solely that of the author.}

\address{Niles Township High School West, Skokie, IL 60077, United States of America}
\email{muhamm1@nilesk12.org}

\address{Euler Circle, Mountain View, CA 94040, United States of America}
\email{hmammarduo@gmail.com}


\begin{abstract}
    Nevanlinna theory studies the value distribution of meromorphic functions and provides powerful results in the form of the First and Second Main Theorems. In this paper, we introduce quaternionic analogues of the Nevanlinna functions. Starting from the Jensen formula due to \cite{perottiJensen2019}, we derive a notion of total order and an associated integrated counting function. We further define quaternionic Weil functions and corresponding mean proximity functions. In this context, we introduce the class of mean proximity balanced functions, which includes the slice-preserving functions and all semiregular functions with a dominating index in their power series. To address the failure of $\log|f^s|$ to be harmonic, we define a Harmonic Remainder Function that compensates for this defect in the Jensen formula. We then prove a weak First Main Theorem--type result for general semiregular functions and obtain a full First Main Theorem for the mean proximity balanced functions.  
\end{abstract}

\maketitle
\tableofcontents

\section{Introduction}

Nevanlinna Theory is the study of the value distribution of meromorphic functions $f:\cc\to\cp$. The First and Second Main Theorems achieve this by relating the growth of $f$ to its zeroes, poles, and order via the characteristic function $T(f,r)$. One of the earliest results in this field is the Little Picard Theorem, which states that any nonconstant entire function of $\cc$ can omit at most one value \cite{Picard1880}. Rolf Nevanlinna developed two powerful generalizations of this statement in \cite{Nevanlinna1925}, known respectively as the First and Second Main Theorems of classical Nevanlinna Theory.

There has been extensive effort in generalizing the First and Second Main Theorem beyond merely the meromorphic functions of $\cc$ and extending them to functions on higher-dimensional complex manifolds and algebraic varieties (see \cite{Bernard84, Griffiths1973, Noguchi2014, Quang2022} for instance). In this paper, we pursue such a generalization in the context of quaternionic analysis, establishing a version of the First Main Theorem with the appropriate notion of meromorphicity. 

Defining holomorphicity and meromorphicity for quaternion-valued functions is nontrivial due to noncommutativity. In the complex case, the various characterizations of holomorphicity, differentiability, satisfaction of the Cauchy--Riemann equations, and analyticity are equivalent. However, for quaternionic functions $f:\mathbb{H}\to\mathbb{H}$, these conditions diverge, and even the naive notion of quaternionic differentiability,
\[
f'(q) \;=\; \lim_{h\to 0} \frac{f(q+h)-f(q)}{h},
\]
forces $f$ to be affine (see \cite{sudbery1979} for further discussion along these lines). We rely on the works of modern quaternionic analysis, initiated by Gentili and Struppa \cite{GENTILI2007279} and developed further by many authors, which introduce the theory of slice regularity. These results develop the appropriate analogues of holomorphic and meromorphic functions, and form the foundation for our discussion in Section \ref{qanal}.

We now provide an overview of the structure of the paper. Sections \ref{sec:Nevanlinna} and \ref{qanal} briefly summarize the relevant background in Nevanlinna Theory and Quaternionic Analysis respectively, and may be omitted by readers already familiar with these topics.

Section \ref{PJSection} introduces the Jensen formula due to \cite{perottiJensen2019}, and provides a few refinements. We then define a unified notion of total multiplicity and spherical order, which we refer to as total order (Definition \ref{totalOrder}). Thus, the Jensen formula can be more cleanly stated as in Theorem \ref{jensen2}.

Section \ref{sec:NevanlinnaFunctions} introduces the four quaternionic Nevanlinna functions considered in this work. We begin with the integrated counting function $N(f,a,r)$ (Definition \ref{def:IntegratedCountingFunction}) which builds on the notion of total order. We characterize this integrated counting function in terms of an unintegrated counting function and then demonstrate the remaining angular dependencies that cannot be resolved with the radially symmetric unintegrated function. Next, we define Weil functions (Definition \ref{def:Weil}) and further a mean proximity function $m(f,a,r)$ (Definition \ref{def:AnalMPF}). Within this framework, we define the class of mean proximity balanced functions, where the mean proximity function behaves compatibly with the spherical conjugate $S_f$. Finally, we define the harmonic remainder function $H(f,a,r)$ (Definition \ref{def:HarmonicRemainderFunction}), which corrects for the failure of $\log|f^s|$ to be harmonic in the Jensen formula, and we combine these constructions to define the quaternionic characteristic function $T(f,a,r)$ (Definition \ref{def:characteristic}). 

Section \ref{sec:FirstMainTheorem} uses the Jensen formula to prove a First Main Theorem. For general semiregular functions, the theorem holds with weak error terms, while for mean proximity balanced functions, it holds with $O(1)$ error, in direct analogy with the classical case. We then establish the algebraic properties of the characteristic function on the mean proximity balanced functions, paralleling those of the complex theory. 

\section{The Nevanlinna Functions and Theorems}\label{sec:Nevanlinna}

For a meromorphic function $f:\cc\to\cp$, \cite{Nevanlinna1925} introduced three fundamental quantities that describe the distribution of values taken by $f$ on $\DR\coloneq \{z\in\cc : |z|<R\}$. 

\begin{defin}[Integrated Counting Function]\label{claIntCou}
    Let $f:\cc\to\cp$ be a meromorphic function on $\DDR$, $R\leq\infty$, and let $a\in\cp$ and $0\leq r\leq R$. Let $n(f,a,r)$ denote the unintegrated counting function defined as the number of times $f$ attains $a$ in $\DDR$, counted with multiplicity. Then, \[N(f,a,r)\coloneq n(f,a,0)\log r+ \int_0^r [n(f,a,t)-n(f,a,0)]\,\frac{dt}{t}\] is the integrated counting function.   
\end{defin}

As opposed to $n(f,a,r)$, $N(f,a,r)$ is a continuous function in $r$ with desirable analytic properties. Unless otherwise stated, we use the term counting function to refer to the integrated counting function. 

\begin{defin}[Mean Proximity Function]
    Let $a\in\cp$, and let $\lambda_a:\cp\setminus\{a\}\to\rr$ be a Weil function, i.e., there exists on every open neighborhood of $a$ a continuous function $\alpha:\cp\to\rr$ such that \[\lambda_a(z)=-\log|z-a|+\alpha(z).\] Let $f:\cc\to\cp$ be a meromorphic function on $\DDR$, $R\leq\infty$. Then for all $r\leq R$, \[m(f,\lambda_a,r)\coloneq\int_0^{2\pi}\lambda_a(f(re^{i\T}))\frac{d\T}{2\pi}\] is a mean proximity function. Conventionally, we choose \[\lambda_a(z)=\begin{cases}
        \log^+\frac1{|z-a|}  &\text{if} \quad  a,z\neq \infty, \\\log^+ |z| &\text{if} \quad a=\infty 
    \end{cases} \quad \text{and} \quad \lambda_a(\infty)=0 \quad \text{if} \quad a\neq\infty.\] We call the mean proximity function generated by this Weil function the analytic mean proximity function, or simply the mean proximity function, denoted by $m(f,a,r)$. 
\end{defin}

We remark that the mean proximity function is a compensatory function, and as such the specific choice of Weil function is not generally important.  

\begin{defin}[Nevanlinna Characteristic Function]
    Let $f$ be meromorphic on $\DDR$, $R\leq\infty$, and let $a\in\cp$. Then for all $r\leq R$, the (analytic) characteristic function is defined by \[T(f,a,r)\coloneq N(f,a,r)+m(f,a,r).\] More generally, for a Weil function $\lambda_a:\cp\setminus\{a\}\to\rr$, the function defined by \[T_{\lambda_a}(f,a,r)=N(f,a,r)+m(f,\lambda_a,r)\] is a Nevanlinna characteristic function. 
\end{defin}

The significance of the characteristic function arises from the fact that it is essentially invariant with respect to the choice of $a\in\cp$, as evidenced in the below theorem. 

\begin{teo}[First Main Theorem]\label{FirstMainTheorem}
    Let $a\in\cp$ and let $f\not\equiv a,\infty$ be a meromorphic function on $\DR$, $R\leq\infty$. Then for all $r\leq R$, \[N(f,a,r)+m(f,a,r)=T(f,r)+O(1),\] where $T(f,r)\coloneq T(f,\infty,r)$.  
\end{teo}
\begin{corr}
    For any $a,b\in\cp$ and $f$ as above, $T(f,a,r)=T(f,b,r)+O(1)$. 
\end{corr}

The First Main Theorem is essentially a generalization of the Fundamental Theorem of Algebra to meromorphic functions, as it gives an upper bound on the number of times $f$ attains $a$. The more difficult lower bound arises from the Second Main Theorem. 

\begin{teo}[Second Main Theorem]
    Let $f$ be a transcendental meromorphic function on $\DR$, $0\leq r\leq R\leq \infty$. For $q\geq 2$, let $a_1,\dots a_q\in\cp$ be $q$ distinct points. Then \[(q-2)T(f,r)\leq N(f,\infty,r)+\sum_{j=1}^q N\left(\frac1{f-a_j},r\right)-N_{ram}(f,r)+o(T(f,r)),\] where $N_{ram}$ is the ramification term, with \[N_{ram}(f,r)=2N(f,\infty,r)-N(f',\infty,r)+N\left(\frac1{f'},\infty,r\right).\] 
\end{teo}

The derivations and proofs of these theorems are covered thoroughly in \cite{cherry2001nevanlinna}. 

\section{Slice Regularity}\label{qanal}

The earliest attempts to define a notion of holomorphicity proceeded by generalizing the Cauchy--Riemann operators. \cite{fueter35} defined a quaternionic function to be \emph{regular} if it solves the equation \[\frac{\partial f}{\partial \conj q}=\frac14\left(\frac{\partial }{\partial x_0}+i\frac{\partial }{\partial x_1}+j\frac{\partial }{\partial x_2}+k\frac{\partial }{\partial x_3}\right)f\equiv 0,\] where $x_0,x_1,x_2,x_3$ are the coordinates of the identification of $\hh$ with $\rr^4$. Fueter-regular functions enjoy many of the key properties of holomorphic functions. Thus, the Cauchy--Riemann system can be replaced by the Cauchy--Fueter system, and the notion of Fueter regularity has been well developed and applied (see \cite{sudbery1979, gurlebeck1990, kravchenko1996} for example). 

However, there are severe limitations to Fueter-regularity that make it a less than desirable generalization of holomorphicity. The strictness of the Cauchy--Fueter condition excludes many desirable functions and in fact does not even include the polynomials\footnote{Due to noncommutativity, we consider the one-sided polynomials.} in the variable $q$. Even the identity function $f(q)=q$ is not Fueter-regular, because $\frac{\partial}{\partial \conj q} q=-\frac12$. \cite{fueter34} attempted to resolve this issue by considering the class of \emph{quaternionic holomorphic functions}, which satisfy Laplace's equation in four real variables \[\frac{\partial}{\partial \conj q}\Delta f(q)=0,\] but this class of functions is extremely large, as it includes the whole class of harmonic functions of four real variables. 

\cite{GENTILI2007279} considered the following decomposition. Note that $\hh = \rr +i\rr + j\rr + k\rr$, and the set of quaternions satisfying $q^2=-1$ forms a 2-sphere. Namely, we define \[\s\coloneq \{q\in\hh : q^2=-1\}.\] More generally, we let $\s_q\coloneq\s_{x+Iy}\coloneq x +y \s$.

This introduces the following interesting geometry. Let $I\in\s$. Considering the set $L_I\coloneq \rr + I\rr$, we remark that $L_I$ can be identified with the complex plane $\cc$, and $\hh = \bigcup_{I\in\s} (\rr + I\rr)$. Since each $L_I$ is isomorphic to $\cc$, we can define a holomorphic derivative on each slice.

\begin{figure}[ht]
\centering
\begin{tikzpicture}[scale=1]
  \def\r{3}
  \node[circle, fill, inner sep=1] (orig) at (0,0) {};
  \def\vx{\r/1.6} \def\vy{\r*0.5}
  \draw[thick,<->] 
    ($(orig) - 1.3*(\vx,\vy)$) -- ($(orig) + 1.3*(\vx,\vy)$) node[right] {$L_I$}; 
  \node[left=45pt, above=15pt] at (orig) {$\s$};
  \fill[gray!20, opacity=0.5] (orig) circle (\r/2);
  \draw (orig) circle (\r/2);
  \draw[dashed] (orig) ellipse (\r/2 and \r/6);

\begin{scope}
  \clip (-\r, -0.7*\r) rectangle (\r, 0.7*\r);
\foreach \y in {-2,-1,...,2} {
  \draw[thin,gray!50] 
    ($(orig) + (-\r,0) + (0,\y*0.3*\r)$) -- ++(2*\r,0);
}

\foreach \x in {-4, -3,-2,...,3, 4} {
  \draw[thin,gray!50] 
    ($(orig) + (2*\r/5,2*\r/3) + (\x*0.3*\r,0)$) -- ++(-4*\r/5,-4*\r/3);
}
\end{scope}

  \draw[<->] ($(orig) + (2*\r/5, 2*\r/3)$) -- ++(-4*\r/5, -4*\r/3) node[below] {$\Im_k(q)$};
  \draw[<->] ($(orig) + (-\r,0)$) -- ++(2*\r, 0) node[right] {$\Im_j(q)$};
  \draw[<->] ($(orig) + (0,-\r)$) -- ++(0, 2*\r) node[above] {$\Im_i(q)$};

\end{tikzpicture}
\caption{The complex line (slice) $L_I$. When restricted to the imaginary axes, $L_I$ is simply a line passing through the imaginary unit $I$ on $\s$. The line functions as the imaginary axis of the full slice.}
\end{figure}
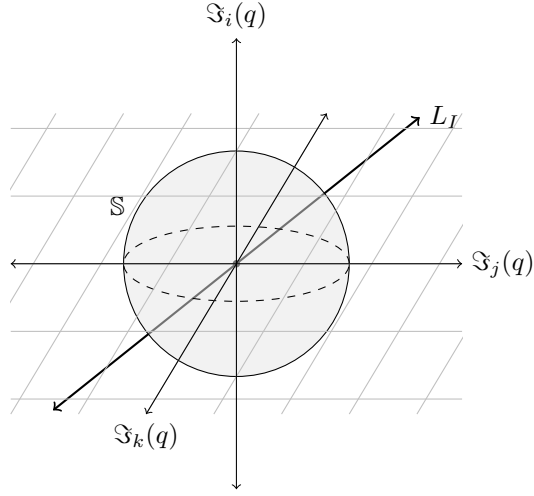

\begin{defin}
    Let $f:\Omega\to\hh$ be a quaternion valued function. For each $I\in\s$, let $\Omega_I=\Omega\cap L_I$ and $f_I=f_{\mid_{\Omega_I}}$ be the restriction of $f$ to $\Omega_I$ so that $f_I:\Omega_I \to \hh$. Then, $f_I$ is holomorphic if \[\conj \partial _I f(x+yI)\coloneq \frac12 \left(\frac{\partial }{\partial x}+I\frac{\partial }{\partial y}\right)f_I(x+yI)\equiv 0.\]
\end{defin}

\begin{defin}
    Let $f:\Omega\to\hh$ be a quaternion valued function. The function $f$ is called (slice-) regular if for every $I\in\s$, $f_I$ is holomorphic. 
\end{defin}
In other words, $f$ is holomorphic when restricted to any slice.

Because slice regularity is a relatively local condition, (in that it is defined slice-wise), there are several immediate pathologies due to lack of continuity across slices. Consider the following example:

\begin{example}[\cite{GENTILI2007279}, Example 1.11]
    Let $I\in\s$ and $f:\hh\setminus\rr\to\hh$ be defined as follows: 
    \[f(q)=\begin{cases}
        0 &\text{if} \quad q\in\hh\setminus L_I \\
        1 &\text{if} \quad q\in L_I \setminus \rr
    \end{cases}.\]
    This function is clearly regular despite not being continuous across slices.  
\end{example}

Fortunately, these issues can be resolved by imposing conditions on the domain. 

\begin{defin}[\cite{GENTILI2007279}, Definition 1.12]
    Let $\Omega$ be a domain in $\hh$. Then $\Omega$ is called a slice domain if it intersects the real axis, and if, for all $I\in\s$, its intersection $\Omega_I$ with the complex plane $L_I$ is connected. 
\end{defin}

This definition essentially forces connectedness of the domain $\Omega$. The second statement guarantees $\Omega$ is connected on any given slice, and to ensure that the slices themselves are connected, we force $\Omega \cap \rr$ to be nonempty, because this is precisely where the slices intersect. 

\begin{defin}
    Let $\Omega$ be a slice domain. If for all points $x+yI\in\Omega$, with $x,y\in \rr$ and $I\in \s$, $\Omega$ contains the whole sphere $x+y\s$, then $\Omega$ is a symmetric slice domain. 
\end{defin}

Symmetric slice domains are often easier to work with, as we are free to choose any member of a sphere $\s_q$ without worrying if it is not contained in the domain. 

The original framework by \cite{GENTILI2007279} has been modified to the alternative $*$-algebras via the notion of \emph{stem functions} and more general \emph{slice functions}. The regular functions defined this way are well-behaved over a larger class of domains. This was originally introduced by \cite{Ghiloni_2011}, though the exposition we provide here is based off of \cite{Perotti_2019} and \cite{Altavilla_2019}. 

Consider the algebra of complex quaternions: \[\hh_\cc\coloneq \hh \otimes_\rr\cc\coloneq \{p+\imath q\mid p,q\in\hh, \imath: \imath^2=-1\}.\] Now let $f:\hh \to \hh$ be a (left) polynomial function defined by $f(q)=\sum_n q^na_n$, with $q, a_i\in\hh$. Now let $z=x+\imath y\in\cc$, and define the lifted polynomial $F:\cc\to\hh_\cc$ by $F(z)=\sum_n z^na_n$. Now define the embedding $\Phi_J:\hh_\cc\to\hh$ by $\Phi_J(x+\imath y)\coloneq x+Jy$ for $J\in\s$. This suggests the following commutative diagram:
\[
\begin{CD}\label{CD}
\cc \simeq \rr \otimes_\rr\cc @>F>> \hh_\cc=\hh\otimes_\rr \cc \\
@V\Phi_j VV   @VV \Phi_j V \\
\hh @>>f> \hh
\end{CD}.
\]

Observe that the polynomial $f$ can be generated by the mapping $F:\cc\to\hh_\cc$. As such, in the more general case, we are motivated to enforce holomorphicity on $F$ to attain a regular function $f$ that is slice-wise regular by the embedding $\Phi_J$. These functions are defined on the following domains:

\begin{defin}
    Let $D\subseteq \cc$ be symmetric to the real axis. We define the circularization of $D$ to be the $\Omega_D\subseteq\hh$ by \[\Omega_D\coloneq \bigcup_{J\in\s}\Phi_J(D)=\{x+Jy \mid x +iy\in D, J\in\s\}.\] Such sets are called circular sets or circular domains. 
\end{defin}

We remark that $\Omega_D$ is symmetric with respect to the real axis, but we do not require $\Omega_D \cap \rr$ to be nonempty. Also, it is not restrictive to have $D$ be symmetric to the real axis, as regardless, $\Omega_D$ will contain $q^c$ for any $q\in\Omega_D$ by the circularization property. 

\begin{defin}
    Let $D\subseteq \cc$ be any symmetric set with respect to the real line. Let $F=F_1+\imath F_2:D \to \hh_\cc$. If $F(\conj z)=\conj{F(z)}$, $F$ is a stem function. 
\end{defin}

\begin{defin}
    Let $f:\Omega_D\to\hh$. We say $f$ is a (left) slice function if it is induced by a stem function $F=F_1+\imath F_2:D \to \hh_\cc$ such that for any $I \in \s$, $x +Iy\in\Omega_D$, \[f(x + Iy)=F_1(x+iy)+IF_2(x +iy).\] The slice function $f$ generated by a stem function $F$ is denoted $\mathcal{I}(F)$.  
\end{defin}

Notice that any quaternion $x+Iy$ can also be written as $x+(-I)(-y)$, so requiring $F(\conj z)=\conj{F(z)}$ ensures the induced slice function is well-defined, independent of the choice of representation. Furthermore, the commutative diagram \ref{CD} holds, with the polynomial $f$ and the lifted polynomial $F$ replaced by the slice function $f$ and the stem function $F$. 

Finally, a function $f=\mathcal{I}(F)$ is (left) regular if its stem function $F$ is holomorphic. In the case where $\Omega_D$ is a slice domain, this definition coincides exactly with that given by \cite{GENTILI2007279}. Note that the family of circular domains contains all symmetric slice domains. From now on, $\Omega_D$ is always taken to be a circular domain. 

We also have a natural definition for \emph{slice derivatives}. 

\begin{defin}
    Let $f$ be a slice function. We define the slice derivative (or merely derivative) $\frac{\partial f}{\partial q}$ and conjugate slice derivative $\frac{\partial f}{\partial \conj q}$ as the slice functions \[\frac{\partial f}{\partial q}\coloneq \mathcal{I}\left(\frac{\partial f}{\partial z}\right) \quad \text{and} \quad \frac{\partial f}{\partial \conj q}\coloneq \mathcal I \left(\frac{\partial F}{\partial \conj z}\right).\] 
\end{defin}

\begin{remark}
    A slice function $f$ is (slice) regular if and only if $\frac{\partial f}{\partial \conj q}\equiv 0$. 
\end{remark}

With this notion of regularity, we can define an appropriate notion of semiregularity generalizing the meromorphic functions on $\cc$. 

\begin{defin}
    A function $f:\Omega_D\to\hh$ is semiregular if it is regular in a symmetric slice domain $\Omega_D' \subseteq \Omega$ such that every point of $\Omega_D\setminus \Omega_D'$ is a pole or removable singularity of $f$. 
\end{defin}
\begin{remark}[\cite{gentili2022}, Remark 5.22]
    If $f$ is semiregular in $\Omega$, then the set of its nonremovable poles consists of isolated real points or isolated spheres. 
\end{remark}
Thus, we operate under the assumption that $\mathcal{P}(f)$ consists of real points and isolated spheres. 

We shall not entertain a full exposition of quaternionic analysis. The interested reader should refer to the book \cite{gentili2022}. 

\subsection{Lemmas on Slice Regularity}

We collect here some definitions and lemmas concerning slice-regular functions as needed throughout the paper.

\begin{defin}
    A slice function $f=\mathcal{I}(F):\Omega_D\to\hh$, $F=F_1+\imath F_2$ is called slice-preserving if $F_1$ and $F_2$ are real valued. Equivalently, $f_I(\Omega_D)\subseteq L_I$ for all $I\in\s$.  
\end{defin}

\begin{defin}
    Let $f,g:\Omega_D\to\hh$ be regular functions induced by $F=F_1+\imath F_2:\cc \to\hh_\cc$ and $G=G_1+\imath G_2:\cc \to\hh_\cc$. Then the $*$-product is the slice function defined by \[f*g\coloneq \mathcal{I}(FG),\] where $FG$ is the associative pointwise product over $\hh_\cc$. Furthermore, $f^{n*}$ refers to the $n$-times product $\underbrace{f*\dots*f}_{n\text{ times}}$.
\end{defin}

It is common to refer to the above operation as the regular product or the slice product as well. The following formula relating the $*$-product and the pointwise quaternionic product is well known. 

\begin{prop}[\cite{gentili2022}, Theorem 3.4]
    Let $f,g:\Omega_D\to\hh$ be regular functions. For all $q\in\Omega_D$, if $f(q)\neq0$, then \[(f*g)(q)=f(q)g(f(q)^{-1}qf(q)).\] If $f(q)=0$, then $(f*g)(q)=0$.  
\end{prop}

It should also be noted that if $f$ is slice-preserving and $g$ is a slice function, then $f*g=fg$, and $g*f=gf$. In other words, the $*$-product coincides with the pointwise product when one of the factors is slice-preserving.

\begin{remark}
    Given any two quaternions $a,b\in\hh$, the element $a^{-1}ba$ belongs to $\s_b$. 
    In other words, conjugation by a quaternion preserves the spherical set of $b$.
\end{remark}

\begin{defin}
    Let $f=\mathcal{I}(F):\Omega_D\to\hh$, with $F=F_1+\imath F_2:\cc \to\hh_\cc$. The conjugate $f^c$ of $f$ is the slice function defined by \[f^c\coloneq \mathcal{I}(F^c)=\mathcal{I}(F_1^c+\imath F_2^c),\] where conjugation is with respect to $\imath$ in $\hh_\cc$. The symmetrization $f^s$ of $f$ is the slice function defined by \[f^s\coloneq f*f^c=f^c*f.\]  
\end{defin}

\begin{remark}
    The symmetrization $f^s$ is always slice-preserving. 
\end{remark}

\begin{defin}
    Let $f:\Omega_D\to\hh$ be a regular function. If $f\not\equiv 0$, the $*$-reciprocal of $f$ is the semiregular function $f^{-*}:\Omega_D\to\hh$ defined by \[f^{-*}=\frac{1}{f^s}f^c.\] Observe that $f*f^{-*}=f^{-*}*f=1$ in $\Omega_D\setminus{\mathcal{P}(f^s)}$. 
\end{defin}

When introducing the stem function framework for slice regularity, \cite{Ghiloni_2011} also introduced the following two operators. 
\begin{defin}
    Let $f:\Omega_D\to\hh$ be a regular function. The function $f^\circ_s : \Omega_D \to \hh$ defined by \[f^\circ_s\coloneq \frac{1}{2}(f(q)+f(\conj q))\] is called the spherical value of $f$. The function $f_s':\Omega_D\setminus\rr\to \hh$ defined by \[f'_s\coloneq \frac12 \Im(q)^{-1}(f(q)-f(\conj q))\] is called the spherical derivative of $f$, where $\Im q$ is the standard imaginary part of $q\in\hh$.  
\end{defin}
Let $q=x +Iy\in\Omega_D$, and $z=x +\imath y\in D$. Then $f^\circ_s(q)=F_1(z)$ and $f'_s(q)=y^{-1}F_2(z)$, where $f=\mathcal I (F)=\mathcal I (F_1+\imath F_2)$. Hence, the two spherical functions defined above are slice functions, and more importantly, are constant on every $\s_q\in\hh$. Moreover, these two functions admit the decomposition of $f$ as \[f(q)=f^\circ_s(q)+\Im (q)f'_s(q).\]

\section{A Quaternionic Jensen Formula}\label{PJSection}

Nevanlinna Theory is underpinned by the Jensen formula, as it provides the foundation for the construction of the Nevanlinna functions. In this section, we introduce several necessary refinements to the known Jensen formula for semiregular functions. 

\begin{teo}[Classical Jensen Formula]
Let $f\not\equiv 0, \infty$ be a meromorphic function on $\DR$. Let $\{a_i\}_{i=1}^p$ denote the zeroes of $f$ in $\DR$, counted with multiplicity, and $\{b_i\}_{i=1}^q$ denote the poles of $f$ in $\DR$, also counted with multiplicity. Then,
\begin{align*}
\log|f(0)| &= \frac{1}{2\pi} \int_0^{2\pi} \log|f(Re^{i\theta})| \, d\theta
- \sum_{i=1}^p \log \frac{R}{|a_i|} 
+ \sum_{i=1}^q \log \frac{R}{|b_i|}.
\end{align*}
\end{teo}

The development of a Jensen formula has been a significant recent pursuit. In \cite{perottiJensen2019}, Perotti derived a general Jensen formula in the context of regular functions, building upon a similar result obtained by \cite{Altavilla_2019}. We begin by presenting this result.

\begin{defin}
    Let $f:\Omega_D\to\hh$ be a nonconstant semiregular function. Let $S_f:\Omega \setminus \mathcal {ZP}(f^s) \to \hh$ be defined by \[S_f(q)\coloneq \begin{cases}
        f'_s(q)f(q)^{-1}\conj q f(q) f'_s(q)^{-1} &\text{if} ~q\not\in \overline{\mathcal {Z}(f'_s)} \\
        \conj q &\text{if} ~q\in\overline{\mathcal {Z}(f'_s)} 
    \end{cases}.\] Note that $S_f$ is a diffeomorphism of $\Omega_D\setminus(\mathcal Z (f^s) \cup \overline{\mathcal Z (f'_s)})$. 
\end{defin}
\begin{remark}
    The set $\mathcal Z (f'_s)$ is typically called the degenerate set of $f$, denoted $D_f$, and represents all the values of $f$ where it is constant over an entire 2-sphere. 
\end{remark}
We refer to the above function as the \emph{spherical conjugate} with respect to $f$. Note that if $f$ is slice-preserving, $S_f(q)=\conj q$ for all $q$. The spherical conjugate admits the following decomposition of $\log|f(z)|$: 

\begin{corr}\label{sphericalConj}
    Let $f$, $S_f$ be defined as above, and let $\mathbb{B}_R\coloneq\mathbb{B}(0,R), R\in(0,\infty)$ be an open ball whose closure is contained in $\Omega_D$. Then, \[\log |f^s(x)|=\log|f(x)|+\log|f(S_f(x))|.\]
\end{corr}
This decomposition allows one to utilize the Jensen formula due to \cite{Altavilla_2019} on the slice symmetric function $f^s$, and yield a formula in terms of $\log|f|$ and $\log|f\circ S_f|$. 

In the Jensen formula below, the spherical conjugate $S_f$ acts as a compensation function that adjusts the boundary integral to match the divisor structure of $f$. 

Finally, we recall the definitions of the isolated and spherical multiplicities from \cite[Definition 3.12]{Stoppato_2012},  total multiplicity from \cite[Definition 6.13]{GENTILI2007279} \footnote{An equivalent and more precise definition for total multiplicity can be found in \cite[Definition 14]{Ghiloni_2011}, but it is less refined.}, order from \cite[Definition 5.18]{GENTILI2007279}, and spherical order from \cite[Definition 5.30]{GENTILI2007279}.  

\begin{teo}[Perotti's Jensen Formula]\label{PJensen}
    Let $\Omega_D$ be an open circular domain in $\hh$, and let the closure of $\mathbb{B}_R$ be contained in $\Omega$. Let $f:\Omega_D\to\hh$ be semiregular and not constant. In $\overline{\mathbb B_R}$, let $\{r_i\}_1^m$  be the isolated real zeroes of $f$, $\{a_i\}_1^t$ and $\{\s_{a_i}\}_{t+1}^l$ be the nonreal isolated and spherical zeroes of $f$,  repeated according to total multiplicity. Let $\{p_i\}_1^n$ be the real poles of $f$, and let $\{\s_{b_i}\}_1^p$ be the spherical poles of $f$ such that all points within it have the same order,  and let $\{\s_{b_i}\}_{p+1}^s$ be the spherical poles of $f$ such that they each contain one point of lesser order, repeated according to spherical order, and let $\{\alpha_i\}_{1}^{s'}$ be the points of lesser order in $\{\s_{b_i}\}_{p+1}^s$ repeated according to isolated multiplicity. Assume that $0$ is neither a pole nor a zero of $f$, and $\partial \mathbb B_R$ does not contain zeroes or poles of $f$. Then, it holds:\footnote{We have abused the index $i$ for the sake of simplicity.} 
    \begin{align*}&\log|f(0)|+\frac{R^2}4\Re \left(\left(f(0)^{-1}\conj{\frac{\partial f}{\partial q}(0)}\right)^2\right)-\frac {R^2}4\Re\left(f(0)^{-1}\frac{\partial^2f}{\partial q^2}(0)\right)\\ &=\frac1{2|\partial \mathbb B_R|}\int_{\partial \mathbb B_R}\log|f(w)|\,d\sigma(w)+\frac1{2|\partial \mathbb B_R|}\int_{\partial \mathbb B_R}\log|(f\circ S_f)(w)|\,d\sigma(w) \\ &\quad -\sum_{i=1}^m\left(\log\frac R{r_i}+\frac{r_i^4-R^4}{4R^2r_i^2}\right) + \sum_{i=1}^n\left(\log\frac R{p_i}+\frac{p_i^4-R^4}{4R^2p_i^2}\right)\\ &\quad-\sum_{i=1}^l\left(2\log\frac R{|a_i|}+\frac{|a_i|^4-R^4}{4R^2|a_i|^4}\left(4(\Re a_i)^2-2|a_i|^2\right)\right)\\&\quad+\sum_{i=1}^s\left(2\log\frac R{|b_i|}+\frac{|b_i|^4-R^4}{4R^2|b_i|^4}\left(4(\Re b_i)^2-2|b_i|^2\right)\right)\\&\quad -\sum_{i=1}^{s'}\left(2\log\frac R{|\alpha_i|}+\frac{|\alpha_i|^4-R^4}{4R^2|\alpha_i|^4}\left(4(\Re \alpha_i)^2-2|\alpha_i|^2\right)\right)
    .\end{align*}
\end{teo}
The formula above is dense, but its intuitive meaning in terms of distribution is clear. Each real zero and pole contributes a term of the form \[\pm\left(\log\frac{R}{\zeta}+\frac{\zeta^4-R^4}{4R^2\zeta^2}\right),\] and each nonreal zero and pole contributes a term of the form \[ \pm\left(2\log\frac{R}{|\zeta|}+\frac{|\zeta|^4-R^4}{4R^2|\zeta|^4}\left(4(\Re \zeta)^2-2|\zeta|^2\right)\right).\]Finally, we consider those spherical poles that have a single point of lower order, and these contribute (against the pure contribution of the spherical poles) \[-\left(2\log\frac{R}{|\zeta|}+\frac{|\zeta|^4-R^4}{4R^2|\zeta|^4}\left(4(\Re \zeta)^2-2|\zeta|^2\right)\right),\] exactly analogous in structure to the contribution of a nonreal zero.  

We correct a subtle inconsistency in the presentation of the above formula when it comes to the treatment of nonreal zeroes and poles, in particular, the extraneous factor of two. This extraneous factor, as we shall see in the next section, causes us to double-count nonreal zeroes and poles.   

\begin{teo}[Jensen Formula]\label{Jensen}
   Let $\Omega_D$, $f$, $R$, $\{r_i\}_1^m$,  $\{a_i\}_1^t$, $\{\s_{a_i}\}_{t+1}^l$, $\{p_i\}_1^n$, $\{\s_{b_i}\}_1^p$, $\{\s_{b_i}\}_{p+1}^s$, $\{\alpha_i\}_{1}^{s'}$ be defined as in Theorem \ref{PJensen}. Assume that $0$ is neither a pole nor a zero of $f$, and $\partial \mathbb B_R$ does not contain zeroes or poles of $f$. Then, it holds:  
   \begin{align*}\log|f(0)|&=\frac{1}{2|\partial \mathbb B_R|}\int_{\partial \mathbb B_R}\log|f(w)|\,d\sigma(w)+\frac{1}{2|\partial \mathbb B_R|}\int_{\partial \mathbb B_R}\log|(f\circ S_f)(w)|\,d\sigma(w)\\  
   &\quad-\frac{R^2}{4}\Re \left(\left(f(0)^{-1}\conj{\frac{\partial f}{\partial q}(0)}\right)^2\right)+\frac {R^2}4\Re\left(f(0)^{-1}\frac{\partial^2f}{\partial q^2}(0)\right) \\
  &\quad -\sum_{i=1}^m\left(\log\frac{R}{r_i}+\frac{r_i^4-R^4}{4R^2r_i^2}\right) + \sum_{i=1}^n\left(\log\frac{R}{p_i}+\frac{p_i^4-R^4}{4R^2p_i^2}\right) 
    \\ &\quad-\sum_{i=1}^l\left(\log\frac{R}{|a_i|}+\frac{|a_i|^4-R^4}{4R^2|a_i|^4}\left(2(\Re a_i)^2-|a_i|^2\right)\right) 
    \\&\quad+\sum_{i=1}^s\left(\log\frac{R}{|b_i|}+\frac{|b_i|^4-R^4}{4R^2|b_i|^4}\left(2(\Re b_i)^2-|b_i|^2\right)\right)
    \\&\quad -\sum_{i=1}^{s'}\left(\log\frac{R}{|\alpha_i|}+\frac{|\alpha_i|^4-R^4}{4R^2|\alpha_i|^4}\left(4(\Re \alpha_i)^2-2|\alpha_i|^2\right)\right).
   \end{align*}
\end{teo}
\begin{proof}
    We modify the proof of \cite[Theorem 3.3]{Altavilla_2019}. As such, we defer the technical details to the cited paper, and only highlight the changes we have made. We assume without loss of generality that $f$ has no real zeroes or poles, because the following correction does not affect those terms. 

    Let \[g(q)\coloneq \left(\prod_{|b_i|< R}B_{\s_{b_i},R}(q)\right)^{-1}\left(\prod_{|a_i|< R}B_{\s_{a_i},R}(q)\right)f^s(q),\] where \[B_{\s_\zeta, \rho}(x)=(\rho^2(x-\zeta)^s)^{-1}(x-\rho^2\zeta^{-1})^s|\zeta|^2.\] Then, by \cite[Theorem 7.24]{mitrea2013}, \cite[Theorem 2.10]{Altavilla_2019}, and because $g$ has no zeroes or poles on $\mathbb B_R$, $\log|g|$ is biharmonic and it holds \[\log|g(0)|=\frac{1}{|\partial \mathbb B_R|}\int_{\partial \mathbb B_R}\log|g(w)|\,d\sigma(w)-\frac{R^2}{8}\Delta\log|g(q)|_{\mid q=0}.\] We have \[\log|g(0)|=\log|f^s(0)|+2\left(\sum_{|\beta_i|<R}\log\frac{R}{|\beta_i|}-\sum_{|\alpha_i|<R}\log\frac{R}{|\alpha_i|}\right),\] and by \cite[Lemma 3.1]{Altavilla_2019} we have \begin{align*}\Delta\log|g(q)|_{|q=0}&=\Delta\log|f^s(q)|_{|q=0}+\sum_{|\alpha_i|<R}\frac{2(R^4-|\alpha_i|^4)}{R^4|\alpha_i|^4}[2|\alpha_i|^2-4(\Re \alpha_i)^2]\\
    &\quad-\sum_{|\beta_i|<R}\frac{2(R^4-|\beta_i|^4)}{R^4|\beta_i|^4}[2|\beta_i|^2-4(\Re \beta_i)^2].\end{align*} Furthermore we have \[\int_{\partial \mathbb B_R}\log|g(w)|\,d\sigma(w)=\int_{\partial \mathbb B_R}\log|f^s(w)|\,d\sigma(w),\] because $B_{\s_\zeta,R}:\partial \mathbb B_R\to\partial \mathbb B_1$. Combining this yields \begin{align*}\log|f^s(0)|&=\frac{1}{|\partial \mathbb B_R|}\int_{\partial \mathbb B_R}\log|f^s(w)|\,d\sigma (w)-\frac{R^2}{8}\Delta\log|f^s(q)|_{\mid q=0}\\ 
    &\quad-\sum_{|a_i|<R}\left(2\log\frac{R}{|a_i|}+\frac{R^2}{8}\frac{2(R^4-|a_i|^4)}{R^4|a_i|^4}\left(2|a_i|^2-4\Re(a_i)^2\right)\right)\\
    &\quad+\sum_{|b_i|<R}\left(2\log\frac{R}{|b_i|}+\frac{R^2}{8}\frac{2(R^4-|b_i|^4)}{R^4|b_i|^4}\left(2|b_i|^2-4\Re(b_i)^2\right)\right).\end{align*} We proceed as in the proof of \cite[Theorem 13]{perottiJensen2019}. Recall Corollary \ref{sphericalConj}, and note that $|f^s(0)|=|f(0)|^2$, so that $\log|f^s(0)|=2\log|f(0)|$. We recall the computation of $\Delta\log|f^s(q)|_{\mid q=0}$ from \cite[Proposition 8]{Altavilla_2019}. Then, we have  
    \begin{align*}\log|f^s(0)|&=\frac{1}{|\partial \mathbb B_R|}\int_{\partial \mathbb B_R}\log|f^s(w)|\,d\sigma (w)-\frac{R^2}{8}\Delta\log|f^s(q)|_{\mid q=0}\\ 
    &\quad-\sum_{|a_i|<R}\left(\log\frac{R}{|a_i|}+\frac{R^2}{8}\frac{2(R^4-|a_i|^4)}{R^4|a_i|^4}\left(|a_i|^2-2\Re(a_i)^2\right)\right)\\
    &\quad+\sum_{|b_i|<R}\left(\log\frac{R}{|b_i|}+\frac{R^2}{8}\frac{2(R^4-|b_i|^4)}{R^4|b_i|^4}\left(|b_i|^2-2\Re(b_i)^2\right)\right).\end{align*}
    The result follows by referring to Theorem \ref{PJensen} if $f$ has real zeroes or poles. 
\end{proof}

\subsection{Total Order}

We now come to the original contributions of this paper. First, we remark that the notions of total multiplicity and the order of the poles as in \ref{Jensen} coincide exactly in the following way: 

\begin{defin}[Total Order]\label{totalOrder}
    Let $f$ be a semiregular function on a circular domain $\Omega_D$ with $f\not\equiv0$. Consider $x+y\s\in\Omega$. There exists $m\in\mathbb{Z}$, $n\in\mathbb{N}$, $p_1,\dots,p_n\in x+y\s$ with $p_i\neq \conj {p_{i+1}}$ for all $i\in\{1,\dots,n-1\}$ so that \[f(q)=[(q-x)^2+y^2]^m(q-p_1)*(q-p_2)*\dots*(q-p_n)*g(q)\] for some semiregular function $g$ on $\Omega_D$ which does not have poles nor zeroes in $x+y\s$. Then, $\ordt_{\s_{x+yI}}(f)\coloneq m+n$ is the total order of $x+y\s$. If $y=0$, then the total order is defined to coincide with the isolated multiplicity of $x$. 

\end{defin}
Note that when $m\geq0$, the above notion of total order is exactly the total multiplicity. Furthermore, with this definition, we may simply count the poles of $f$ in Theorem \ref{Jensen} according to total order, instead of counting according to the spherical order and correcting for the contributions due to isolated points in these spherical poles. We also note that the total order is \emph{signed}, so that poles have negative total order. In this way, both the zeroes and poles of $f$ can be counted according to total order in Theorem \ref{Jensen}. Thus, total order is a natural generalization of the existing notion of total multiplicity.\footnote{We later discovered the exposition of \cite[Section 8]{Bisi2021} which provides a related but distinct counting notion using \emph{divisors}. In particular, see \cite[Proposition 8.8]{Bisi2021}.}

We pose an equivalent definition of total order, which aligns more closely with the modern definition of total multiplicity in terms of the classical order of the symmetrization. 

\begin{defin}[Total Order]\label{totalOrderr}
    Let $f$ be a semiregular function on a circular domain $\Omega_D$ with $f\not\equiv0$. For any $\zeta\in\Omega_D$, there exists $k\in\mathbb Z$ so that $[(q-\zeta)^s]^{-k}f^s(q)$ has no zeroes or poles in $\s_\zeta$. We say $f$ has total order $k$ at $\s_\zeta$, and we denote $k$ by $\ordt_{\s_\zeta}(f)$. 
\end{defin}

The following remark demonstrates that the kernel counting the real zeroes (resp. real poles) and the kernel counting the nonreal zeroes (resp. nonreal poles) are in fact the same. 

\begin{remark}
    Let $r\in\hh\cap\rr$. Then, \[\left(\log\frac{R}{|r|}+\frac{|r|^4-R^4}{4R^2|r|^4}\left(2(\Re r)^2-|r|^2\right)\right)=\left(\log\frac{R}{r}+\frac{r^4-R^4}{4R^2r^2}\right).\]
\end{remark}

This is a trivial computation. In other words, Theorem \ref{Jensen} counts each nonreal zero (resp. nonreal pole) according to the kernel \[J(\zeta,R)\coloneq\left(\log\frac{R}{|\zeta|}+\frac{|\zeta|^4-R^4}{4R^2|\zeta|^4}\left(2(\Re \zeta)^2-|\zeta|^2\right)\right),\] while each real zero (resp. real pole) is also counted according to the same kernel. 

Finally, we define \[
\mathcal S(\Omega_D) \coloneq 
\Bigl\{ \s_\zeta : \zeta \in \Omega_D, \;\bigl(\exists q \in \s_\zeta \text{ s.t. } f(q) = 0 \text{ or } \infty \implies f(\zeta) = 0 \text{ or } \infty\bigr) \Bigr\}.
\]
to denote the set of all 2-spheres contained in $\Omega_D$. The second condition is merely to ensure that if $\zeta$ is an isolated nonreal zero (resp. nonreal pole), that the sphere $\s_\zeta$ is in fact indexed by $\zeta$ (and not a nonzero [resp. nonpole]) in the set $\mathcal{S}(\Omega_D)$. With this, we can state a much cleaner version of Theorem \ref{Jensen}. 

\begin{teo}[Jensen Formula with Total Order]\label{jensen2}
     Let $\Omega_D$ be an open circular domain in $\hh$, and let the closure of $\mathbb{B}_R$ be contained in $\Omega_D$. Let $f:\Omega_D\to\hh$ be semiregular and not constant. Assume that $0$ is neither a pole nor a zero of $f$, and $\partial \mathbb B_R$ does not contain zeroes or poles of $f$. Then, it holds:
     \begin{align*}
    \log|f(0)| &= \frac{1}{2|\partial \mathbb B_R|}\int_{\partial \mathbb B_R}\log|f(w)|\,d\sigma(w)+\frac{1}{2|\partial \mathbb B_R|}\int_{\partial \mathbb B_R}\log|(f\circ S_f)(w)|\,d\sigma(w) \\  
   &\quad  -\frac{R^2}{4}\Re \left(\left(f(0)^{-1}\conj{\frac{\partial f}{\partial q}(0)}\right)^2\right)+\frac {R^2}4\Re\left(f(0)^{-1}\frac{\partial^2f}{\partial q^2}(0)\right) \\
    &\quad -\sum_{\substack{
        \s_\zeta\in\mathcal{S}(\mathbb B_R) \\
        \zeta\neq 0
    }}\ordt_{\s_\zeta}(f)J(\zeta,R).\end{align*}
\end{teo}

\begin{remark}\label{zeroHandleRemark}
    In the summations of Theorem \ref{jensen2}, we make the assumption that $0$ is not a zero or pole of $f$. It is easy to see why this is the case; the Jensen Kernel is not even defined for $\zeta=0$. Thus, as in the classical case, we consider the case where $0$ is possibly a zero or pole separately, by letting $\log|f(q)|=m\log |q|+\log|g(q)|$, if $f(q)=q^mg(q)$. 
\end{remark}

\section{Nevanlinna Functions}\label{sec:NevanlinnaFunctions}
We aim to package the ``unrefined" terms of Theorem \ref{jensen2} to create suitable Nevanlinna functions, in analogy with the complex case. 

\subsection{Integrated Counting Function}\label{sec:IntegratedCountingFunction}
The goal of this subsection is to define a notion of \textit{counting} from the summation terms of Theorem \ref{jensen2}. 

 For the sake of notational convenience, we define\footnote{Note the identity \[(f^{-*})^s=(f^s)^{-1},\] and so we may use Definition \ref{totalOrderr}}\[\ordt_{\s_q}(f,a)\coloneq \max\{0,\ordt_{\s_q}(f-a)\}, \quad \ordt_{\s_q}(f,\infty)\coloneq \max\{0,\ordt_{\s_q}(f^{-*})\}.\] We remark that $\ordt_{\s_q}(f)$ is signed, as is the Jensen order, while $\ordt_{\s_q}(f,a)$ for any $a\in\hp$ is always nonnegative. Conceptually, the question that the notation $\ordt_{\s_q}(f,a)$ answers is: with what multiplicity does $f$ attain $a$ on the sphere $\s_q$? We prefer to use the signed definition when poles contribute negatively in a summation, while we prefer to use the nonnegative definition in most other scenarios. This formulation leads to our first definition of the integrated counting function. 

\begin{table}[ht]
\centering
\caption{Notation summary for Section \ref{sec:IntegratedCountingFunction}.}
\label{tab:counting-notation}
\begin{tabularx}{\textwidth}{@{}lX@{}}
\toprule
Symbol & Meaning \\ 
\midrule
$\ordt_{\s_q}(f)$ & (Signed) total order of $f$ on  $\s_q$ \\[2pt]
$\ordt_{\s_\zeta}(f,a)$ & Total order of $f-a$ on $\s_\zeta$; unsigned total order \\[2pt]
$n(f,a,r)$ & Unintegrated counting function: $\displaystyle \sum_{\s_\zeta\subset\overline{\mathbb B_r}}\ordt_{\s_\zeta}(f,a)$ \\[2pt]
$N(f,a,r)$ & Integrated counting function, with both radial and angular contributions \\[2pt]
$\A(f,a,r)$ & Angular counting term, depending on $\Re \zeta$ \\[2pt]
$a_r(f,a,t)$ & Angular unintegrated count: $\displaystyle \sum_{\substack{\s_\zeta\subset\overline{\mathbb B_r}\\|\Re\zeta|\le t}}\ordt_{\s_\zeta}(f,a)$ \\[2pt]
$a_r^{\Re}(f,a,t)$ & Weighted angular count: $\displaystyle \sum_{\substack{\s_\zeta\subset\overline{\mathbb B_r}\\|\Re\zeta|\le t}}\ordt_{\s_\zeta}(f,a)\,(\Re\zeta)^2$ \\
\bottomrule
\end{tabularx}
\end{table}

\begin{defin}[Integrated Counting Function]\label{def:IntegratedCountingFunction}
      Let $f:\Omega_D\to\hp$ be semiregular, and let the closure of $\mathbb{B}_r$ be contained in $\Omega_D$. Then, \[N(f,a,r)\coloneq n(f,a,0)\log r+\sum_{\substack{ \s_\zeta\in\mathcal{S}(\overline{\mathbb B_r}) \\ \zeta\neq 0}}\ordt_{\s_\zeta}(f,a)J(\zeta,r).\]
\end{defin}
This definition arises from applying Theorem \ref{jensen2} to the function $f-a$ (resp. $f^{-*}$). The term $n(f,a,0)\log r$ arises by applying what was discussed in Remark \ref{zeroHandleRemark}. 

In analogy with the complex case, we desire to be able to define the integrated counting function in terms of an unintegrated counting function. This is fairly simple in the complex case, as the counting kernel in the classical Jensen formula allows a simple decomposition via a radially symmetric integral. In the quaternionic case, we must deal with both radial and angular parts, which leads to further dependencies. 

\begin{defin}[Unintegrated Counting Function]
     Let $f:\Omega_D\to\hp$ be semiregular, and let the closure of $\mathbb{B}_r$ be contained in $\Omega_D$. We define the unintegrated counting function $n(f,a,r)$ as the number of times $f$ attains $a$ in $\overline{\mathbb B_r}$ repeated according to total order. Formally, \[n(f,a,r)\coloneq \sum_{\s_{\zeta}\in\mathcal S(\overline{\mathbb B_r})}\ordt_{\s_\zeta}(f,a).\]
\end{defin}
Observe that $n(f,a,r)$ is a nondecreasing step function. 

\begin{prop}\label{quatIntCou2}
    Let $f$ be defined as above. Then, \begin{align*}N(f,a,r)&=n(f,a,0)\log r + \int_0^r[n(f,a,t)-n(f,a,0)]
    \,\frac{dt}{t} \\&\quad+\sum_{\substack{ \s_\zeta\in\mathcal{S}(\overline{\mathbb B_r}) \\ \zeta\neq 0}}\ordt_{\s_\zeta}(f,a) \frac{|\zeta|^4-r^4}{4r^2|\zeta|^4}\left(2(\Re \zeta)^2-|\zeta|^2\right).\end{align*}
\end{prop}
\begin{proof}
    We decompose $K(\zeta,r)$ into its radial and nonradial parts as \begin{align*}
        N(f,a,r)&=n(f,a,0)\log r+\sum_{\substack{ \s_\zeta\in\mathcal{S}(\overline{\mathbb B_r}) \\ \zeta\neq 0}}
       \ordt_{\s_\zeta}(f,a)\log\frac{r}{|\zeta|} \\ 
       &\quad+ \sum_{\substack{ \s_\zeta\in\mathcal{S}(\overline{\mathbb B_r}) \\ \zeta\neq 0}}\ordt_{\s_\zeta}(f,a)\frac{|\zeta|^4-r^4}{4r^2|\zeta|^4}\left(2(\Re \zeta)^2-|\zeta|^2\right). 
    \end{align*}
    
    Then, simply observe that
      \begin{align*}
    \sum_{\substack{ \s_\zeta\in\mathcal{S}(\overline{\mathbb B_r}) \\ \zeta\neq 0}}
       \ordt_{\s_\zeta}(f,a)\log\frac{r}{|\zeta|} 
    &= \sum_{\substack{ \s_\zeta\in\mathcal{S}(\overline{\mathbb B_r}) \\ \zeta\neq 0}}
       \ordt_{\s_\zeta}(f,a)\int_{|\zeta|}^r\frac{dt}{t} \\
    &= \int_0^r \frac{1}{t}\sum_{\substack{ \s_\zeta\in\mathcal{S}(\overline{\mathbb B_t}) \\ \zeta\neq 0}}
       \ordt_{\s_\zeta}(f,a)\,dt \\
       &\quad \\
    &= \int_0^r [n(f,a,t)-n(f,a,0)]\,\frac{dt}{t},
  \end{align*} where it is justified to switch the order of integration and summation because the summation is finite. \footnote{What we mean by this is that we can only have finitely many contributions from zeroes and poles in $\mathbb{B}_r$. Though we sum over all spheres in $\mathcal{S}(\mathbb B_r)$, all but the zeroes and poles contribute trivially to the sum.}
\end{proof}

\begin{remark}
    Note that the integrated counting function as in \ref{quatIntCou2} is exactly analogous to the classical integrated counting function as in Definition \ref{claIntCou}, with the addition of the terms involving the summand \[\ordt_{\s_\zeta}(f,a)\frac{|\zeta|^4-r^4}{4r^2|\zeta|^4}(2(\Re \zeta)^2-|\zeta|^2),\] which arise from the harmonic remainder of the Blaschke factors. 
\end{remark}

We may further extract a radial term from the remaining summation. 

\begin{prop}\label{quaIntCou3}
    Let $f$ be defined as in Proposition \ref{quatIntCou2}. Then, \begin{align*}
        N(f,a,r) &= n(f,a,0)\log r + \int_0^r[n(f,a,t)-n(f,a,0)]\,\frac{dt}{t} \\
        &\quad +  \int_0^r\frac{t^4+r^4}{2r^2t^3}[n(f,a,t)-n(f,a,0)]\,dt \\ 
        &\quad + \sum_{\substack{ \s_\zeta\in\mathcal{S}(\overline{\mathbb B_r}) \\ \zeta\neq 0}}
        \ordt_{\s_\zeta}(f,a) \frac{|\zeta|^4-r^4}{2r^2|\zeta|^4}(\Re \zeta)^2
    \end{align*}
\end{prop}
\begin{proof}
    We begin by writing \begin{align*} \sum_{\substack{ \s_\zeta\in\mathcal{S}(\overline{\mathbb B_r}) \\ \zeta\neq 0}}\ordt_{\s_\zeta}(f,a)\frac{|\zeta|^4-r^4}{4r^2|\zeta|^4}(2(\Re \zeta)^2-|\zeta|^2) &= -\frac{1}{4}\sum_{\substack{ \s_\zeta\in\mathcal{S}(\overline{\mathbb B_r}) \\ \zeta\neq 0}}\ordt_{\s_\zeta}(f,a)\frac{|\zeta|^4-r^4}{r^2|\zeta|^2}\\ &\quad+\sum_{\substack{ \s_\zeta\in\mathcal{S}(\overline{\mathbb B_r}) \\ \zeta\neq 0}}
        \ordt_{\s_\zeta}(f,a) \frac{|\zeta|^4-r^4}{2r^2|\zeta|^4}(\Re \zeta)^2.\end{align*} It suffices to look at \[\sum_{\substack{ \s_\zeta\in\mathcal{S}(\overline{\mathbb B_r}) \\ \zeta\neq 0}}\ordt_{\s_\zeta}(f,a)\frac{|\zeta|^4-r^4}{r^2|\zeta|^2}=\sum_{\substack{ \s_\zeta\in\mathcal{S}(\overline{\mathbb B_r}) \\ \zeta\neq 0}}\ordt_{\s_\zeta}(f,a)\frac{|\zeta|^2}{r^2}-\sum_{\substack{ \s_\zeta\in\mathcal{S}(\overline{\mathbb B_r}) \\ \zeta\neq 0}}\ordt_{\s_\zeta}(f,a)\frac{r^2}{|\zeta|^2},\] and we can analyze the terms on the right-hand side separately. For the first, we write \begin{align}\label{quaIntCou3.1}
            \sum_{\substack{ \s_\zeta\in\mathcal{S}(\overline{\mathbb B_r}) \\ \zeta\neq 0}}\ordt_{\s_\zeta}(f,a)\frac{|\zeta|^2}{r^2} 
            &= \frac{1}{r^2} \sum_{\substack{ \s_\zeta\in\mathcal{S}(\overline{\mathbb B_r}) \\ \zeta\neq 0}}\ordt_{\s_\zeta}(f,a) |\zeta|^2 \nonumber\\
            &= \frac{1}{r^2} \sum_{\substack{ \s_\zeta\in\mathcal{S}(\overline{\mathbb B_r}) \\ \zeta\neq 0}}\ordt_{\s_\zeta}(f,a)\left(r^2+\int_{|\zeta|}^r-2t \,dt \right) \nonumber\\
            &= \sum_{\substack{ \s_\zeta\in\mathcal{S}(\overline{\mathbb B_r}) \\ \zeta\neq 0}}\ordt_{\s_\zeta}(f,a) +\frac{1}{r^2}\sum_{\substack{ \s_\zeta\in\mathcal{S}(\overline{\mathbb B_r}) \\ \zeta\neq 0}}\ordt_{\s_\zeta}(f,a)\int_{|\zeta|}^r-2t \,dt \nonumber\\ 
            &= [n(f,a,r)-n(f,a,0)]-\frac{1}{r^2}\int_0^r 2t \left(\sum_{\substack{ \s_\zeta\in\mathcal{S}(\overline{\mathbb B_t}) \\ \zeta\neq 0}}\ordt_{\s_\zeta}(f,a)\right)\, dt \nonumber\\
            &= [n(f,a,r)-n(f,a,0)]-\frac{1}{r^2}\int_0^r 2t [n(f,a,t)-n(f,a,0)] \,dt 
        \end{align}

    Analogously for the second term, \begin{align}\label{quaIntCou3.2}
        \sum_{\substack{ \s_\zeta\in\mathcal{S}(\overline{\mathbb B_r}) \\ \zeta\neq 0}}\ordt_{\s_\zeta}(f,a)\frac{r^2}{|\zeta|^2} 
        &= r^2 \sum_{\substack{ \s_\zeta\in\mathcal{S}(\overline{\mathbb B_r}) \\ \zeta\neq 0}}\ordt_{\s_\zeta}(f,a) \frac{1}{|\zeta|^2} \nonumber\\
        &= r^2 \sum_{\substack{ \s_\zeta\in\mathcal{S}(\overline{\mathbb B_r}) \\ \zeta\neq 0}}\ordt_{\s_\zeta}(f,a)\left(\frac{1}{r^2}+\int_{|\zeta|}^r\frac{2}{t^3}\,dt \right) \nonumber\\
        &= \sum_{\substack{ \s_\zeta\in\mathcal{S}(\overline{\mathbb B_r}) \\ \zeta\neq 0}}\ordt_{\s_\zeta}(f,a)+ r^2\sum_{\substack{ \s_\zeta\in\mathcal{S}(\overline{\mathbb B_r}) \\ \zeta\neq 0}}\ordt_{\s_\zeta}(f,a)\int_{|\zeta|}^r\frac{2}{t^3}\,dt \nonumber\\
        &= [n(f,a,r)-n(f,a,0)]+r^2\int_0^r \frac{2}{t^3}\left(\sum_{\substack{ \s_\zeta\in\mathcal{S}(\overline{\mathbb B_t}) \\ \zeta\neq 0}}\ordt_{\s_\zeta}(f,a)\right) \,dt \nonumber\\
        &= [n(f,a,r)-n(f,a,0)]+r^2\int_0^r \frac{2}{t^3}[n(f,a,t)-n(f,a,0)] \,dt.
    \end{align}
    Subtracting Equation \ref{quaIntCou3.2} from \ref{quaIntCou3.1}, recalling Proposition \ref{quatIntCou2}, and noting \[\frac{t}{2r^2}+\frac{r^2}{2t^3}=\frac{t^4+r^4}{2r^2t^3}\] yields the desired result after simplification. We note again that in the above, switching the order of summation and integration is justified due to the summation being over finitely many terms. 
\end{proof}

We are now left with a term that cannot be further simplified purely in terms of the radial unintegrated counting function, due to the factor of $\Re\zeta $. 
\begin{defin}[Angular Counting Term]\label{AngularCountingTerm}
     Let $f:\Omega_D\to\hp$ be semiregular, and let the closure of $\mathbb{B}_r$ be contained in $\Omega_D$. Then, we define the angular counting term as \[ \A(f,a,r)\coloneq \sum_{\substack{ \s_\zeta\in\mathcal{S}(\overline{\mathbb B_r}) \\ \zeta\neq 0}}
        \ordt_{\s_\zeta}(f,a) \frac{|\zeta|^4-r^4}{2r^2|\zeta|^4}(\Re \zeta)^2.\]
\end{defin}

We can utilize the trivial bound $|\Re \zeta|\leq |\zeta|$ to obtain a radial estimate of the angular counting term. 

\begin{prop}
    Let $f$ be defined as in \ref{AngularCountingTerm}. Then, \begin{align*}
    \A(f,a,r)\leq -\int_0^r\frac{t^4+r^4}{r^2t^3}[n(f,a,t)-n(f,a,0)]\,dt.
    \end{align*}
\end{prop}
\begin{proof}
    This follows directly from the proof of Proposition \ref{quaIntCou3}.
\end{proof}

Another approach is to define a nonradial unintegrated counting function that also includes angular dependencies. In particular, the following definition is useful. 

\begin{defin}[Angular Unintegrated Counting Functions]
    Let $f$ be defined as in \ref{AngularCountingTerm}. We define \[a_r(f,a,t)\coloneq \sum_{\substack{ \s_\zeta\in\mathcal{S}(\overline{\mathbb B_r}) \\ \zeta\neq 0 \\|\Re \zeta|\leq t}}\ordt_{\s_\zeta}(f,a).\] In other words, $a_r(f,a,t)$ counts the number of times $f$ attains $a$ in $\overline{\mathbb{B}_r}$ according to total order, with the additional condition that the real part of $\zeta$ is at most $t$. We also define \[a_r^\Re (f,a,t)\coloneq \sum_{\substack{ \s_\zeta\in\mathcal{S}(\overline{\mathbb B_r}) \\ \zeta\neq 0 \\|\Re \zeta|\leq t}}\ordt_{\s_\zeta}(f,a)(\Re \zeta)^2.\]
\end{defin}
\begin{remark}
    The difference in the above two representations is as follows. In $a_r(f,a,t)$, we absorb the angular dependency into the summation itself, at the cost of losing radial symmetry in the summation. In $a_r^\Re(f,a,t)$, we preserve the angular dependency in the summand, but lose a summation purely over $\ordt(f,a)$. The exponent of two in $(\Re\zeta)^2$ is merely to align with the structure of the angular counting term. 
\end{remark}

The following proposition provides an exact analytic quantification of the Angular Counting Term, including dependence on the angular unintegrated counting functions. 

\begin{prop}
    Let $f, \A$ be defined as in \ref{AngularCountingTerm}. Then, \begin{align*}
        \A(f,a,r)&=4r^2\iint_{0\leq h\leq t\leq r} ht^{-5}[a_t(f,a,h)-a_t(f,a,0)]\, dh\, dt \\&\quad - \int_0^r 2r^2t^{-3}[n(f,a,t)-n(f,a,0)]\,dt. \\
        &= -2r^2\int_0^rt^{-5}[a_r^\Re(f,a,t)-a_r^\Re(f,a,0)]\,dt.  \\
    \end{align*}
\end{prop}
\begin{proof}
    The proof is much the same as in Propositions \ref{quatIntCou2} and \ref{quaIntCou3}. We again begin by decomposing \begin{align}\label{quaIntCou4.1}\sum_{\substack{ \s_\zeta\in\mathcal{S}(\overline{\mathbb B_r}) \\ \zeta\neq 0}}\ordt_{\s_\zeta}(f,a) \frac{|\zeta|^4-r^4}{2r^2|\zeta|^4}(\Re \zeta)^2 
    &=\sum_{\substack{ \s_\zeta\in\mathcal{S}(\overline{\mathbb B_r}) \\ \zeta\neq 0}}\ordt_{\s_\zeta}(f,a)\frac{(\Re \zeta)^2}{2r^2} \nonumber\\
    &\quad- \sum_{\substack{ \s_\zeta\in\mathcal{S}(\overline{\mathbb B_r}) \\ \zeta\neq 0}}\ordt_{\s_\zeta}(f,a)\frac{(\Re \zeta)^2r^2}{2|\zeta|^4} \nonumber\\
    &= \sum_{\substack{ \s_\zeta\in\mathcal{S}(\overline{\mathbb B_r}) \\ \zeta\neq 0}}\ordt_{\s_\zeta}(f,a)\frac{(\Re \zeta)^2}{2r^2} \nonumber\\
    &\quad- \frac{r^2}{2}\sum_{\substack{ \s_\zeta\in\mathcal{S}(\overline{\mathbb B_r}) \\ \zeta\neq 0}}\ordt_{\s_\zeta}(f,a)\left(\frac{(\Re \zeta)^2}{r^4}+\int_{|\zeta|}^r\frac{4(\Re \zeta)^2}{t^5}\,dt\right) \nonumber\\
    &=-\frac{r^2}{2}\sum_{\substack{ \s_\zeta\in\mathcal{S}(\overline{\mathbb B_r}) \\ \zeta\neq 0}}\ordt_{\s_\zeta}(f,a)\int_{|\zeta|}^r\frac{4(\Re \zeta)^2}{t^5}\,dt, 
    \end{align}
so it suffices to look at the remaining term. We have \begin{align}\label{quaIntCout4.3}
    2r^2\sum_{\substack{ \s_\zeta\in\mathcal{S}(\overline{\mathbb B_r}) \\ \zeta\neq 0}}\ordt_{\s_\zeta}(f,a)\int_{|\zeta|}^r\frac{(\Re \zeta)^2}{t^5} &= 2r^2\int_0^r t^{-5}\left(\sum_{\substack{ \s_\zeta\in\mathcal{S}(\overline{\mathbb B_t}) \\ \zeta\neq 0}}\ordt_{\s_\zeta}(f,a)(\Re \zeta)^2\right) \, dt.
\end{align}
Observe that we can treat the inner term in a similar manner, in that we attempt to further write the term $(\Re \zeta)^2$ as an integral. We analyze this term independently, and thus we have \begin{align}\label{quaIntCou4.2}
\sum_{\substack{ \s_\zeta\in\mathcal{S}(\overline{\mathbb B_t}) \\ \zeta\neq 0}}\ordt_{\s_\zeta}(f,a)(\Re \zeta)^2 
&= \sum_{\substack{ \s_\zeta\in\mathcal{S}(\overline{\mathbb B_t}) \nonumber \\ \zeta\neq 0}}\ordt_{\s_\zeta}(f,a)\left(t^2 -\int_{|\Re \zeta|}^t2h \,dh\right) \\
&= t^2 \sum_{\substack{ \s_\zeta\in\mathcal{S}(\overline{\mathbb B_t}) \nonumber\\ \zeta\neq 0}}\ordt_{\s_\zeta}(f,a)-2\int_0^t h\left( \sum_{\substack{ \s_\zeta\in\mathcal{S}(\overline{\mathbb B_t}) \\ \zeta\neq 0 \\|\Re \zeta|\leq h}}\ordt_{\s_\zeta}(f,a)\right)\,dh \\ 
&=t^2[n(f,a,t)-n(f,a,0)]-2\int_0^t h [a_t(f,a,h)-a_t(f,a,0)]\,dh.
\end{align}
Thus, we have \begin{align*}
      2r^2\sum_{\substack{ \s_\zeta\in\mathcal{S}(\overline{\mathbb B_r}) \\ \zeta\neq 0}}\ordt_{\s_\zeta}(f,a)\int_{|\zeta|}^r\frac{(\Re \zeta)^2}{t^5} 
      &= 2r^2\int_0^rt^{-3}[n(f,a,t)-n(f,a,0)]\, dt \\
      &\quad- 4r^2\int_0^rt^{-5}\left(\int_0^th[a_t(f,a,h)-a_t(f,a,0)]\,dh \right)\,dt. 
\end{align*}
The final integral term on the right-hand side cannot be expressed as a single variable, because $a_t(f,a,h)$ depends on both $t$ and $h$, though we may still apply Fubini's theorem. The first equality follows by recalling equations \ref{quaIntCou4.1} and \ref{quaIntCou4.2}. 

The second equality comes from returning to Equation \ref{quaIntCout4.3}, where we recall the definition of $a_r^\Re (f,a,t)$ and note \[\sum_{\substack{ \s_\zeta\in\mathcal{S}(\overline{\mathbb B_t}) \\ \zeta\neq 0}}\ordt_{\s_\zeta}(f,a)(\Re \zeta)^2=a_r^\Re (f,a,t)-a_r^\Re (f,a,0).\] 

\end{proof}

The above computations with the integrated counting function are summarized in the below final proposition. 

\begin{prop}[Analytic Characterization of the Integrated Counting Function]
    Let $f:\Omega_D\to\hp$ be semiregular, and let the closure of $\mathbb B_r$ be contained in $\Omega_D$. Then, \begin{align*}
        N(f,a,r)&=n(f,a,0)\log r+\int_0^r [n(f,a,t)-n(f,a,0)]\,\frac{dt}{t} \\
        &\quad + \int_0^r \frac{t^4+r^4}{2r^2t^3}[n(f,a,t)-n(f,a,0)]\,dt \\ 
        &\quad +\iint_{0\leq h\leq t\leq r} 4r^2ht^{-5}[a_t(f,a,h)-a_t(f,a,0)]\, dh\, dt \\
        &\quad -\int_0^r 2r^2t^{-3}[n(f,a,t)-n(f,a,0)]\,dt \\\\
        &=n(f,a,0)\log r+\int_0^r [n(f,a,t)-n(f,a,0)]\,\frac{dt}{t} \\
        &\quad +\int_0^r \frac{t^4+r^4}{2r^2t^3}[n(f,a,t)-n(f,a,0)]\,dt \\
        &\quad -\int_0^r2r^2t^{-5}[a_r^\Re(f,a,t)-a_r^\Re(f,a,0)]\,dt.
        \\\\
        &\leq n(f,a,0)\log r+\int_0^r [n(f,a,t)-n(f,a,0)]\,\frac{dt}{t} \\
        &\quad - \int_0^r \frac{t^4+r^4}{2r^2t^3}[n(f,a,t)-n(f,a,0)]\,dt. 
    \end{align*}
\end{prop}

\subsection{Mean Proximity Functions}
We now consider the integrals in Theorem \ref{jensen2}. We may desire that the integrals over $\log|f|$ and $\log|f\circ S_f|$ are in some sense, nearly identical. Unfortunately, this is not the case. Our first proposition shows that the integrals are not generally equal up to an additive constant, which is the strongest relation one could reasonably expect. 

\begin{prop}\label{mIntC_fProp}
   Let $f:\Omega_D\to \hp $ be semiregular and nonconstant, and let $\overline{\mathbb B_r}\subseteq \Omega_D$. Then, in general, \[\frac{1}{|\partial \mathbb B_r|}\int_{\partial \mathbb B_r}\log|f(w)|\,d\sigma(w)\neq \frac{1}{|\partial \mathbb B_r|}\int_{\partial \mathbb B_r}\log |(f\circ S_f)(w)|\,d\sigma(w)+C_f.\] where $C_f$ is a constant with respect to $r$ depending on $f$. 
   \end{prop}
   
   \begin{remark}
    Theorem \ref{FirstMainTheorem} holds for a constant depending on $f$ and $a$, but not on $r$. Proposition \ref{mIntC_fProp} shows that if one ignores the integral over $f\circ S_f$, such a constant independent of $r$ cannot, in general, be achieved.
    \end{remark}
       
    \begin{proof}
    The proof of this fact essentially comes down to the fact that a semiregular function $f$ is not necessarily log-biharmonic (see \cite[Remark 2.8]{Altavilla_2019} for a more detailed discussion on this). Assume to the contrary that \begin{equation}\label{eq:IntEqualityCont}\frac{1}{|\partial \mathbb B_r|}\int_{\partial \mathbb B_r}\log|f(w)|\,d\sigma(w)= \frac{1}{|\partial \mathbb B_r|}\int_{\partial \mathbb B_r}\log |(f\circ S_f)(w)|\,d\sigma(w)+C_f.\end{equation}
    Now choose $f$ without zeroes and poles in $\mathbb B_r$. Then, recalling Theorem \ref{jensen2}, we have \begin{align}\label{eq:IntEquality2}
        \log|f(0)|&=\frac1{|\partial \mathbb B_r|}\int_{\partial \mathbb B_r}\log|f(w)|\,d\sigma(w) -\frac{r^2}{4}\Re \left(\left(f(0)^{-1}\conj{\frac{\partial f}{\partial q}(0)}\right)^2\right)\nonumber\\
        &\quad +\frac {r^2}4\Re\left(f(0)^{-1}\frac{\partial^2f}{\partial q^2}(0)\right) + C_f \nonumber\\
        &= \frac1{|\partial \mathbb B_r|}\int_{\partial \mathbb B_r}\log|f(w)|\,d\sigma(w)-\frac{r^2}{16}\Delta_4\log |f^s(q)|_{\mid q=0}+C_f
    \end{align}
    where in the last line, we have undone the calculation of the Laplacian in \cite[Proposition 8]{perottiJensen2019}. Now consider the expansion of spherical mean values (see \cite[Theorem 3]{Ovall2016}) \[\frac1{|\partial \mathbb B_r|}\int_{\partial \mathbb B_r}\log|f(w)|\,d\sigma(w)=\log|f(0)|+\frac{r^2}{8}\Delta_4\log|f(q)|_{\mid q=0}+O(r^4)\] and \[\frac1{|\partial \mathbb B_r|}\int_{\partial \mathbb B_r}\log|(f\circ S_f)(w)|\,d\sigma(w)=\log|(f\circ S_f)(0)|+\frac{r^2}{8}\Delta_4\log|(f\circ S_f)(q)|_{\mid q=0}+O(r^4).\] Applying Equation \ref{eq:IntEqualityCont} implies \[\log|f(0)|+\frac{r^2}{8}\Delta_4\log|f(q)|_{\mid q=0}=\log|(f\circ S_f)(0)|+\frac{r^2}{8}\Delta_4\log|(f\circ S_f)(q)|_{\mid q=0}+O(r^4).\] We note that the $O(r^4)$ terms are negligible for sufficiently small $r$. Consequently, \[\Delta_4\log|f(q)|_{\mid q=0}=\Delta_4\log|(f\circ S_f)(q)|_{\mid q=0}.\] Hence, returning to Equation \ref{eq:IntEquality2}, and recalling Equation \ref{sphericalConj} and the linearity of the Laplacian,  we have \[\log|f(0)|= \frac1{|\partial \mathbb B_r|}\int_{\partial \mathbb B_r}\log|f(w)|\,d\sigma(w)-\frac{r^2}{8}\Delta_4\log |f(q)|_{\mid q=0}+C_f.\] But this is just the biharmonic mean value identity (with the addition of the $C_f$ term).\footnote{This can only hold with an $O(r^4)$ error term.}) But a general semiregular function $f$ need not be log-biharmonic. Hence, we have a contradiction. 
    \end{proof}
    \begin{remark}
    The proof of Proposition \ref{mIntC_fProp} demonstrates that the class of functions for which the integrals over $f$ and $f\circ S_f$ are equivalent up to a constant is exactly the log-biharmonic one, i.e., the slice-preserving one. 
    \end{remark}
                                               
The natural question is whether Proposition \ref{mIntC_fProp} holds if we weaken the $C_f$ term to be any $O(1)$ error term. This turns out to be a difficult question to answer directly. Instead, we shall consider the subset $\mathcal{MPB}(\Omega_D)\subseteq \mathcal {SR}(\Omega_D)$ for which the integrals are equivalent up to $O(1)$ \footnote{We introduce this notation for clarity in the present definition, but do not use $\mathcal{SR}(\Omega_D)$ or $\mathcal{MPB}(\Omega_D)$ elsewhere.}. 

\begin{defin}\label{def:MPB(Omega)}
Let $\mathcal{SR}(\Omega_D)$ be the set of all semiregular functions on $\Omega_D$. Let $f\in \mathcal{SR}(\Omega_D):\Omega_D\to \hp $ be semiregular and nonconstant, and let $\overline{\mathbb B_r}\subseteq \Omega_D$. Let $R\coloneq \sup \{r>0\mid \overline{\mathbb B_r}\subseteq \Omega_D\}$. Let $\mathcal{R}\coloneq \{r\in(0,R)\mid \partial \mathbb B_r\cap \mathcal{ZP}(f)=\varnothing\}$. Then, we define 
\[
\mathcal{MPB}(\Omega_D) \coloneq
\left\{
f \in \mathcal{SR}(\Omega_D) \;\middle|\;
\begin{aligned}
&\forall a \in \hh,\; \forall r \in \mathcal R,\\
&\frac{1}{|\partial \mathbb B_r|}
\int_{\partial \mathbb B_r}
\log\left|
\frac{f(w)}{(f\circ S_{f-a})(w)}
\right|\, d\sigma(w)
= O(1)
\end{aligned}
\right\}.
\]
  
\end{defin}

We shall refer to such functions as mean proximity balanced functions. We first note that slice-preserving functions are mean proximity balanced (in fact, up to equality), despite the fact that $f-a$ fails to be slice-preserving (and thus log-biharmonic) in general.

\begin{prop}
    Let $f$ be semiregular and slice-preserving on $\Omega_D$. Then, $f$ is mean proximity balanced. 
\end{prop}
\begin{proof}
    One proof follows from the fact that $(f-a)^s=(f-a)^2$, and using \cite[Proposition 8]{perottiJensen2019}. A more insightful proof is as follows. 

    Assume without loss of generality that $f$ does not have a zero or pole at $q=0$, so that $f$ admits a power series. We have $f(q)=\sum_{n\in\mathbb N}a_nq^n$, and $S_{f-a}(q)=(f-a)^{-1}\conj q (f-a)$.\footnote{Suppose $q$ is outside the radius of convergence of the power series. In this case, we may take $q_0\not\in\mathcal{ZP}(f)$ close to $q$, and utilize the power series expansion centered at $q_0$.} Hence, \[(f\circ S_{f-a})(q)=\sum_{n\in \mathbb N}a_n ((f-a)^{-1}\conj q (f-a))^n.\] By the identity $(x^{-1}yx)^n=x^{-1}y^nx$, we have \[(f\circ S_{f-a})(q)=\sum_{n\in \mathbb N}a_n (f-a)^{-1}(\conj q)^n (f-a).\] The power series coefficients $a_n$ are real because $f$ is slice-preserving, and thus commute. Hence, \begin{align}
        (f\circ S_{f-a})(q)&=(f-a)^{-1}\left(\sum_{n\in\mathbb N}a_n (\conj q)^n\right) (f-a) \\
        &= (f-a)^{-1}f(\conj q)(f-a).
    \end{align}
    Hence, $|(f\circ S_{f-a})(q)|=|f(\conj q)|$, and consequently, \[\frac{1}{|\partial \mathbb B_r|}\int_{\partial \mathbb B_r}\log|(f\circ S_{f-a})(q)|\,d\sigma (q)=\frac{1}{|\partial \mathbb B_r|}\int_{\partial \mathbb B_r}\log|f(q)|\,d\sigma (q),\] because $q \mapsto\conj q$ is a measure-preserving isometry of $\partial\mathbb B_r$. 
\end{proof}

\begin{prop}
Let $f$ be semiregular on $\Omega_D$, and let $g,h$ be regular on $\Omega_D$ so that $f=g/h$. Let \[ g(q)=\sum_{n\in\mathbb{N}} a_nq^n, \quad h(q)=\sum_{m\in\mathbb N}b_m q^m.\] Let $n_0$ and $m_0$ be unique dominating indices so that \[|a_{n_0}|R^{n_0}\gg\sum_{n\in\mathbb{N}\setminus\{n_0\}} |a_n|R^n, \quad |b_{m_0}|R^{m_0}\gg\sum_{m\in\mathbb{N}\setminus\{m_0\}} |b_m|R^m.\] Then, $f$ is mean proximity balanced.
\end{prop}
\begin{proof}

The hypothesis guarantees \[f(q)=\frac{a_{n_0}}{b_{m_0}}q^{n_0-m_0}(1+o(1)),\] as $|q|=R\to \infty$. Hence, \[\log|f(q)|=(n_0-m_0)\log|q|+\log\left|\frac{a_{n_0}}{b_{m_0}}\right|+o(1).\] One easily confirms $|S_{f-a}(q)|=|q|$ by its definition. Note that the asymptotic expansion for $\log|f(q)|$ depends only on $|q|$; thus the same expansion applies to $\log|f(S_{f-a}(q))|$. Hence, we have \begin{align*}\log|f(S_{f-a}(q))|&=(n_0-m_0)\log|S_{f-a}(q)|+\log\left|\frac{a_{n_0}}{b_{m_0}}\right|+o(1)\\ 
&=(n_0-m_0)\log|q|+\log\left|\frac{a_{n_0}}{b_{m_0}}\right|+o(1).\end{align*} 

Thus, \[\log|f(S_{f-a}(q))|-\log|f(q)|=o(1)\] uniformly for $|q|=R$. Consequently,\footnote{The dominance assumption implies that the above $o(1)$ term is uniform on $\partial \mathbb B_R$. Since $S_f(\partial \mathbb B_R)=\partial \mathbb B_R$, the same uniform bound holds for $\log|f(S_{f-a}(q))|$. In particular, we have \[\limsup_{|q|=R, R\to \infty}|\log|f(S_{f-a}(q))|-\log|f(q)||=0.\]} \[\left|\frac{1}{|\partial \mathbb{B}_R|}\int_{\partial \mathbb{B}_R}\left(\log|f(S_{f-a}(q))|-\log|f(q)|\right)\,d\sigma(q)\right|=o(1),\] and in particular, \[\left|\frac{1}{|\partial \mathbb{B}_R|}\int_{\partial \mathbb{B}_R}\left(\log|f(S_{f-a}(q))|-\log|f(q)|\right)\,d\sigma(q)\right|=O(1).\]

\end{proof}

A sharp characterization of mean proximity balanced functions, as well as the question of whether every semiregular function enjoys this property, remains unresolved in the present work.

Having settled this point, we proceed to define proximity functions. As in the classical case, we first define Weil functions in our context. 

\begin{defin}\label{def:Weil}
    Let $a\in\hp$. A Weil function with a singularity at $a$ is a continuous map $\lambda_a:\hp \setminus\{a\}\to \rr$ such that in some open neighborhood $U$ of $a\in\hp$, there exists a continuous function $\alpha$ on $U\subseteq\hh$ such that $\lambda_a(q)=-\log|q-a|+\alpha(q)$. 
\end{defin}

Note that $q$ here is a local coordinate, so in a neighborhood of $q=\infty$, we instead look at $\lambda_a(q^{-1})$. 

\begin{remark}\label{remark:WeilO(1)}
    As in the classical case, the difference between any two Weil functions with the same singular point $a$ is bounded due to the compactness of $\hp$. This affords us the convenience of choosing suitable Weil functions to achieve differing error terms.
\end{remark}

Consequently, we have the corresponding mean proximity function. 

\begin{defin}[Mean Proximity Function]
Let $f$ be semiregular on $\Omega_D$. Then, \[m(f,\lambda_a,r)\coloneq \frac{1}{|\partial \mathbb B_r|}\int_{\partial \mathbb B_r}\lambda_a(f(q))\,d\sigma(q).\]
\end{defin}

\begin{remark}\label{remark:m(f,lambda_a,r)O(1)}
Let $\lambda_a$ and $\tilde\lambda_a$ be two Weil functions with the same singularity $a$. 
It follows from Remark \ref{remark:WeilO(1)} that
\[
m(f,\lambda_a,r)-m(f,\tilde\lambda_a,r)=O(1),
\]
uniformly for all $r\in\mathcal R$. Thus, the mean proximity function is well-defined up to an $O(1)$ term.
\end{remark}

In view of this, our results will be independent of choice of Weil function up to an $O(1)$ term. We therefore fix the Weil function used by \cite{Nevanlinna1925} for the remainder of this work. 

\begin{defin}[Analytic Mean Proximity Function]\label{def:AnalMPF}
Let $f$ be semiregular on $\Omega_D$. Let 
\[\lambda_a(q)=\begin{cases}
        \log^+\frac1{|q-a|}  &\text{if} \quad  a,q\neq \infty, \\\log^+ |q| &\text{if} \quad a=\infty 
    \end{cases} \quad \text{and} \quad \lambda_a(\infty)=0 \quad \text{if} \quad a\neq\infty,\] where $\log^+x=\max\{0,\log x\}$.
Then, \[m(f,a,r)\coloneq \frac{1}{|\partial \mathbb B_r|}\int_{\partial \mathbb B_r}\lambda_a(f(q))\,d\sigma(q).\]
\end{defin}

One of the primary motivations for utilizing $\log^+$ over $\log$ is that the former ensures that ``distance," in the sense of proximity, is always nonnegative. 

\subsection{Harmonic Remainder Function}

This is the only elementary Nevanlinna function that does not have an analogue in the complex case. It comes from the term \[-\frac{r^2}{4}\Re \left(\left(f(0)^{-1}\conj{\frac{\partial f}{\partial q}(0)}\right)^2\right)+\frac {r^2}4\Re\left(f(0)^{-1}\frac{\partial^2f}{\partial q^2}(0)\right),\] which is equivalent to (see \cite[Proposition 8]{perottiJensen2019}) \[-\frac{r^2}{16}\Delta_4\log |f^s(q)|_{\mid q=0}.\] Because this term contributes an error of $O(r^2)$, it cannot be ignored without weakening the resulting First Main Theorem. 

\begin{defin}\label{def:HarmonicRemainderFunction}
     Let $f$ be semiregular on $\Omega_D$, and let $a\in\hp$. Then for all $r\in \mathcal{R}$, we define \[H(f,a,r)\coloneq \frac{r^2}{16}\Delta_4\log \left|(f(q)-a)^s\right|_{\mid q=0},\] and $H(f,\infty,r)\equiv0$.
\end{defin}

The Harmonic Remainder Function corrects for the failure of $\log|f^s|$ to be harmonic, which leads to a Laplacian correction term in the Jensen formula. It corrects the discrepancy between Theorem \ref{jensen2} applied to \(f\) and \(f-a\). For \(a=\infty\), no discrepancy arises because Jensen is applied directly to \(f\), so we set \(H(f,\infty,r)\equiv0\). 

Unlike the case of the Mean Proximity Function, discussion on the dependence of the Harmonic Remainder Function on the symmetrization $f^s$ is uninteresting. Indeed, by \cite[Proposition 8]{perottiJensen2019} the Laplacian term may be written entirely in terms of $f$ and its slice derivatives at the center. Thus, $H(f,a,r)$ is an intrinsic function of $f$, $a$, and $r$. 

\subsection{Characteristic Function}

The final function to be defined is the analogue of the characteristic function. There are two natural definitions: one adapted to mean proximity balanced functions, and one valid in general depending on the symmetrization $f^s$. For the mean proximity balanced functions, these two definitions differ by at most $O(1)$, which is absorbed by the First Main Theorem regardless. As such, we adopt the general definition throughout this work.\footnote{For mean proximity balanced functions, a more natural definition is simply $T(f,a,r)=N(f,a,r)+m(f,a,r)-H(f,a,r)$.} 

\begin{defin}[Nevanlinna Characteristic Function]\label{def:characteristic}
    Let $f$ be semiregular on $\Omega_D$, and let $a\in\hh$. Then for all $r\in \mathcal{R}$, \[T(f,a,r)\coloneq N(f,a,r)+\frac{1}{2}m((f-a)^s, 0, r)-H(f,a,r).\] If $a=\infty$, then \[T(f,r)\coloneq T(f,\infty,r)= N(f,\infty, r)+\frac{1}{2}m(f^s,\infty,r),\] recalling that $H(f,\infty,r)\equiv 0$. 
\end{defin}

The well-definedness of $T(f,r)$ is the subject of Theorem \ref{teo:FMT}.  

In terms of notation, we adopt the conventions of \cite{cherry2001nevanlinna}. Note that due to Remark \ref{remark:m(f,lambda_a,r)O(1)}, the characteristic is well-defined up to $O(1)$, even if utilizing a different Weil function. 

A few of the algebraic properties of the characteristic function carry over from the classical case, with modifications due to noncommutativity. Unlike the classical case, many of these identities can be proved only in the case of $a=\infty$. However, the transport of Theorem \ref{teo:FMT} extends these properties to arbitrary $a$.\footnote{Of note, this transport is $O(1)$ only in the case of mean proximity balanced functions. In the general case, one achieves much weaker identities due to worse error terms.}

Notably, subadditivity over addition involves a more complicated nontrivial mixed proximity term.\footnote{This is ultimately related to the possibility of functions not being mean proximity balanced.}

\begin{prop}\label{prop:CharacteristicAlgera}
    Let $f,g$ be semiregular on $\Omega_D$. 
Then for all $r\in\mathcal R$,
\begin{align}
    T(f^{n*},\infty,r)&=nT(f,\infty,r) \label{eq:1CharacteristicAlgera}\\
    T(f*g,\infty,r) &\leq T(f,\infty,r)+T(g,\infty,r)  \label{eq:2CharacteristicAlgera}\\ 
    T(f+g,\infty,r)&\leq T(f,\infty,r)+T(g,\infty,r)+\log 3 \label{eq:3CharacteristicAlgera}\\
    &\quad +\frac{1}{2}m(f*g^c+g*f^c,\infty, r) \notag\\
    T(f^c,a,r)&=T(f,a,r) = \frac{1}{2}T(f^s,\infty,r). \label{eq:4CharacteristicAlgera}
  \end{align}
\end{prop}

\begin{remark}
By induction, the inequalities in Proposition \ref{prop:CharacteristicAlgera} extend to finite \(*\)-products and finite sums of semiregular functions.
\end{remark}

\begin{remark}
    The elementary techniques hold only at $a=\infty$ for the simple reason that terms like $(f*g-a)^s$ and $(f+g-a)^s$ cannot be dealt with without significant mixed terms. When we can have an elementary identity for generic $a$, we state it in the proof below. 
    
     In Equation \ref{eq:3CharacteristicAlgera}, the mixed term disappears for mean proximity balanced Functions (see Proposition \ref{prop:CharacteristicAlgebraMPB}). 
\end{remark}

\begin{proof}
    For Equation \ref{eq:1CharacteristicAlgera}, note that $\ordt_{\s_\zeta}\left(f^{n*},a\right)=n\ordt_{\s_\zeta}(f,a)$, and $(f^{n*})^s=(f^s)^{n*}=(f^s)^n$, and hence the result follows. For Equation \ref{eq:2CharacteristicAlgera}, note that $\ordt_{\s_\zeta}(f*g,a)=\ordt_{\s_\zeta}(f,a)+\ordt_{\s_\zeta}(g,a)$, and \[(f*g)^s=(f*g)*(g^c*f^c)=f*g^s*f^c=fg^sf^c,\] because $g^s$ is slice-preserving. Now note \begin{align*}
        (f*g)^s*f&=fg^sf^c*f =fg^sf^s \\
        (f*g)^sf&=fg^sf^s \\
        (f*g)^s&=fg^sf^sf^{-1}\\
        |(f*g)^s|&=|f||g^s||f^s||f|^{-1}\\
        &=|f^s||g^s|,
    \end{align*} where we have used the fact that $(f*g)^s$ is slice-preserving, and the multiplicativity of the norm. Recalling $\log^+|xy|\leq \log^+|x|+\log^+|y|$, the result follows. 

    For Equation \ref{eq:3CharacteristicAlgera}, we deal with the proximity terms first. We have \((f+g)^s=f^s+f*g^c+g*f^c+g^s.\) By the triangle inequality, we have \[|(f+g)^s|\leq |f^s|+|g^s|+|f*g^c+g*f^c|.\] We cannot further bound the mixed terms. Then, applying $\log^+$ and recalling $\log^+|x+y+z|\leq \log^+|x|+\log^+|y|+\log^+|z|+\log 3$ yields, \[m((f+g)^s,\infty,r)\leq m(f^s,\infty,r)+m(g^s,\infty,r)+m(f*g^c+g*f^c,\infty,r)+\log 3.\] $N(f+g,a,r)$ is not trivially bounded above in terms of $N(f,a,r)+N(g,a,r)$. However, $f+g$ can only have a pole if $f$ or $g$ has a pole. Thus, $N(f+g,\infty,r)\leq N(f,\infty, r)+N(g,\infty,r)$, which yields the desired result. 

    Finally, Equation \ref{eq:4CharacteristicAlgera} follows from the identities \(\ordt_{\s_\zeta}(f^c,a)=\ordt_{\s_\zeta}(f,a)=\tfrac12\ordt_{\s_\zeta}(f^s,a)\), together with \((f^c-a)^s=((f-a)^c)^s=(f-a)^s\) and \((f^s-a)^s=(f^s-a)^2\).
\end{proof}

\section{A First Main Theorem}\label{sec:FirstMainTheorem}

We are now prepared to state a First Main Theorem derived from Theorem \ref{jensen2} and the discussions in Section \ref{sec:NevanlinnaFunctions}. 

\begin{teo}[First Main Theorem]\label{teo:FMT}
    Let $a\in\hp$ and let $f\not\equiv a,\infty$ be semiregular on $\Omega_D$. Then, for all $r\in\mathcal R$, 
    \begin{equation}\label{eq:FMT1} N(f,a,r)+\frac{1}{2}m((f-a)^s,0,r)-H(f,a,r)=T(f,r)+O(m(ff^c,\infty,r))+O(1).\end{equation} 
    It also holds,
    \begin{align}\label{eq:FMT2} N(f,a,r)+\frac{1}{2}m(f,a,r)+\frac{1}{2}m(f\circ S_{f-a},a,r)-H(f,a,r)&=T(f,r)+O(1) \\&\quad -\frac{1}{2}m(f\circ S_f,\infty,r) \notag\\
    &\quad+\frac{1}{2}m(f\circ S_{f-a},\infty,r)\notag\end{align}
    Moreover, if $f$ is a mean proximity balanced function, then it holds 
    \begin{equation}\label{eq:FMT3} N(f,a,r)+m(f,a,r)-H(f,a,r)=T(f,r)+O(1). \end{equation}
\end{teo}

\begin{remark}\label{remark:FMT1}
    The coefficient on $m(ff^c,\infty,r)$ in Equation \ref{eq:FMT1} is only in the interval $[-1,1]$. We use $O$-notation for convenience in writing.   
\end{remark}

\begin{remark}
    The correction terms on the right-hand side of Equation \ref{eq:FMT2} are a forced normalization. They demonstrate that the obstruction in achieving a full First Main Theorem is precisely comparing the terms $m(f\circ S_f,\infty,r)$ and $m(f\circ S_{f-a},\infty,r)$, the essential problem being that the latter term remains dependent on $a$, despite measuring proximity to $\infty$. We find that the identity \[S_{f-a}=[f'_s(1-f^{-1}a)^{-1}f^{'-1}_s] S_f[f'_s(1-f^{-1}a)f^{'-1}_s]\] provides useful intuition, by demonstrating that $S_{f-a}$ is a conjugation of $S_f$. This does not, however, resolve the issue of dependence on $a$. 
\end{remark}

\begin{remark}
    The left-hand sides of Equations \ref{eq:FMT1} and \ref{eq:FMT2} each motivate two different definitions for the characteristic function $T(f,a,r)$. These definitions are equivalent up to $O(1)$, because \[\log^+|f^s|\leq \log^+|f|+\log^+|f\circ S_f|\leq \log^+|f^s|+\log 2,\] with the corresponding identity on the mean proximity function following similarly. Thus, we do not require a secondary definition. 
\end{remark}

\begin{proof}
    Let $h(q)=f(q)-a$, and apply Theorem \ref{jensen2} to $h$. If $f(0)=a$, then refer to Remark \ref{zeroHandleRemark}. Then, by Definition \ref{def:IntegratedCountingFunction} and Definition \ref{def:HarmonicRemainderFunction}, we have \begin{align}\label{eq:FMT4}
        \log|f(0)-a|&=\frac{1}{2|\partial\mathbb B_r|}\int_{\partial\mathbb B_r}\log|f(q)-a|\,d\sigma(q) + \frac{1}{2|\partial\mathbb B_r|}\int_{\partial\mathbb B_r}\log|(f\circ S_{f-a})(q)-a|\,d\sigma(q)\\
        &\quad - N(f,a,r)+N(f,\infty,r) -H(f,a,r) \notag,
    \end{align}
    by counting zeroes and poles. 

    Recalling $\log|q|=\log^+|q|-\log^+\left|\frac{1}{q}\right|$, we have by Definition \ref{def:AnalMPF} \begin{align*}
        \log|f(0)-a|&=\frac{1}{2|\partial\mathbb B_r|}\int_{\partial\mathbb B_r}\log^+|f(q)-a|\,d\sigma(q)-\frac{1}{2|\partial\mathbb B_r|}\int_{\partial\mathbb B_r}\log^+\left|\frac{1}{f(q)-a}\right|\,d\sigma(q) \\
        &\quad+\frac{1}{2|\partial\mathbb B_r|}\int_{\partial\mathbb B_r}\log^+|(f\circ S_{f-a})(q)-a|\,d\sigma(q)-\frac{1}{2|\partial\mathbb B_r|}\int_{\partial\mathbb B_r}\log^+\left|\frac{1}{(f\circ S_{f-a})-a}\right|\,d\sigma(q) \\
        &\quad - N(f,a,r)+N(f,\infty,r) -H(f,a,r) \\
        &=\frac{1}{2}m(f-a,\infty,r)-\frac{1}{2}m(f,a,r)+\frac{1}{2}m(f\circ S_{f-a}-a,\infty,r)+\frac{1}{2}m(f\circ S_{f-a},a,r) \\
        &\quad - N(f,a,r)+N(f,\infty,r) -H(f,a,r) .\end{align*}

    By the elementary identity $\log|x\pm y|\leq \log|x|+\log|y|+\log 2$, we have \[m(f-a,\infty,r)=m(f,\infty,r)+O(1), \quad m(f\circ S_{f-a}-a,\infty,r)=m(f\circ S_{f-a},\infty,r)+O(1).\] Hence, \begin{align*}
        O(1)&= \frac{1}{2}m(f,\infty,r)-\frac{1}{2}m(f,a,r)+\frac{1}{2}m(f\circ S_{f-a},\infty,r)+\frac{1}{2}m(f\circ S_{f-a},a,r) \\
        &\quad- N(f,a,r)+N(f,\infty,r) -H(f,a,r), 
    \end{align*}
    and this proves Equation \ref{eq:FMT2}. 

    Now recall \cite[Proposition 8]{perottiJensen2019} once again, so Equation \ref{eq:FMT4} becomes \begin{align*}
         \log|f(0)-a|&=\frac{1}{2|\partial\mathbb B_r|}\int_{\partial\mathbb B_r}\log|(f-a)^s(q)|\,d\sigma(q) \\
        &\quad - N(f,a,r)+N(f,\infty,r) -H(f,a,r) \notag.
    \end{align*} Again, by Definition \ref{def:AnalMPF} and the elementary identity  $\log|q|=\log^+|q|-\log^+\left|\frac{1}{q}\right|$, we have 
    \begin{align*}
        \log|f(0)-a|&=\frac{1}{2|\partial\mathbb B_r|}\int_{\partial\mathbb B_r}\log^+|(f-a)^s(q)|\,d\sigma(q)-\frac{1}{2|\partial\mathbb B_r|}\int_{\partial\mathbb B_r}\log^+\left|\frac{1}{(f-a)^s(q)}\right|\,d\sigma(q) \\
        &\quad - N(f,a,r)+N(f,\infty,r) -H(f,a,r) \\
        &=\frac{1}{2}m((f-a)^s,\infty,r)+\frac{1}{2}m((f-a)^s,0,r) \\
        &\quad - N(f,a,r)+N(f,\infty,r) -H(f,a,r) .\end{align*}
The remaining obstruction is dealing with the term $m((f-a)^s,\infty,r)$. Observe that \begin{align*}
    \log^+|(f-a+a)^s|&\leq \log^+|(f-a)^s|+\log^+|f|+\log^+|f^c|+\log^+|a^4|+\log 4 \\
    \log^+|(f-a)^s|&\leq \log^+|f^s|+\log^+|f|+\log^+|f^c|+\log^+|a^4|+\log 4.
\end{align*}
Hence, 
\begin{align*}
    \log^+|(f-a+a)^s|&\leq \log^+|(f-a)^s|+\log^+|ff^c|+\log^+|a^4|+\log6\\
    \log^+|(f-a)^s|&\leq \log^+|f^s|+\log^+|ff^c|+\log^+|a^4|+\log 6.
\end{align*}
And so, we have $m((f-a)^s,\infty,r)=m(f^s,\infty,r)+O(m(ff^c,\infty,r))+O(1)$, where as noted in Remark \ref{remark:FMT1}, the coefficient absorbed into the $O$-notation is in $[-1,1]$. This proves Equation \ref{eq:FMT1}. 

Finally, note that if $f$ is mean proximity balanced, then Theorem \ref{jensen2} becomes \begin{align*}
    O(1) &= \frac{1}{|\partial \mathbb B_R|}\int_{\partial \mathbb B_R}\log|f(w)|\,d\sigma(w) \\  
   &\quad -N(f,a,r)+N(f,\infty,r)-H(f,a,r), \end{align*} where we have again used Definitions \ref{def:IntegratedCountingFunction} and \ref{def:HarmonicRemainderFunction} to write the correction terms. The proof here thus follows the classical case. Using $\log|q|=\log^+|q|-\log^+\left|\frac{1}{q}\right|$ and $\log|x\pm y|\leq \log|x|+\log|y|+\log 2$ as before, we have \[O(1)=m(f,\infty,r)-m(f,a,r)-N(f,a,r)+N(f,\infty,r)-H(f,a,r),\] and Equation \ref{eq:FMT3} follows.   
\end{proof}

We now turn our attention to the additional algebraic properties of the characteristic function on mean proximity balanced functions, as many of the obstructions observed in Proposition \ref{prop:CharacteristicAlgera} disappear. 

\begin{prop}\label{prop:CharacteristicAlgebraMPB}
     Let $f,g$ be semiregular and mean proximity balanced on $\Omega_D$, and let $a,b\in\hp$. Let $\Phi$ be a fractional linear transform with 
\[
\Phi(q) = (Aq + B)*(Cq + D)^{-*},
\quad
\begin{pmatrix}
A & B\\
C & D
\end{pmatrix}
\in GL(2,\mathbb H).
\]
Then for all $r\in\mathcal R$, 
  \begin{align}
    T(f,a,r)&=T(f,b,r)+O(1)  \label{eq:1CharactersticAlgebraMPB}\\
    T(f^{n*},a,r)&=nT(f,a,r)+O(1) \label{eq:2CharactersticAlgebraMPB}\\
    T(f*g,a,r)&\leq T(f,a,r)+T(g,a,r)+O(1) \label{eq:3CharactersticAlgebraMPB}\\
    T(f+g,a,r)&=T(f,a,r)+T(g,a,r)+O(1) \label{eq:4CharactersticAlgebraMPB}\\
    T(f^{-*},a,r)&=T(f,a,r)+O(1)  \label{eq:5CharactersticAlgebraMPB}\\
    T(\Phi(f),a,r) &= T(f,a,r)+O(1)  \label{eq:6CharactersticAlgebraMPB}
\end{align}
\end{prop}

\begin{proof}
Equation \ref{eq:1CharactersticAlgebraMPB} follows by applying Theorem \ref{teo:FMT} to $T(f,a,r)$ and $T(f,b,r)$ separately. Equations \ref{eq:2CharactersticAlgebraMPB} and \ref{eq:3CharactersticAlgebraMPB} follow by applying Equation \ref{eq:1CharactersticAlgebraMPB} to Equations \ref{eq:1CharacteristicAlgera} and \ref{eq:2CharacteristicAlgera}. 

For Equation \ref{eq:4CharactersticAlgebraMPB}, recall \[\log^+|f^s|\leq \log^+|f|+\log^+|f\circ S_f|\leq \log^+|f^s|+\log 2,\] so $m((f-a)^s,0,r)=2m(f,a,r)+O(1)$. Hence, \[m\bigg(\big((f+g)-a\big)^s,0,r\bigg)=2m(f+g,a,r)+O(1).\] 

Now take $a=\infty$. Then, by the identity $\log^+|x+y|\leq \log^+|x|+\log^+|y|+\log 2$, we have $m(f+g,a,r)\leq m(f,a,r)+m(g,a,r)+\log 2$. As noted in the proof of Proposition \ref{prop:CharacteristicAlgera}, $f+g$ can only have a pole if $f$ or $g$ has a pole. Hence, $N(f+g,\infty,r)\leq N(f,\infty, r)+N(g,\infty,r)$, and we have $T(f+g,\infty,r)\leq T(f,\infty,r)+T(g,\infty,r)+\log 2$. Using Equation \ref{eq:1CharactersticAlgebraMPB} to transport to arbitrary $a$ yields the desired result. 

For Equation \ref{eq:5CharactersticAlgebraMPB}, note that $N(f^{-*},0,r)=N(f,\infty,r)$, and $m(f^{-*},0,r)=m(f,\infty,r)$. Hence, $T(f^{-*},0,r)=T(f,r)$. By Equation \ref{eq:1CharactersticAlgebraMPB}, $T(f^{-*},0,r)=T(f^{-*},a,r)+O(1),$ and recalling Theorem \ref{teo:FMT} to rewrite $T(f,r)$ yields the desired result.

For Equation \ref{eq:6CharactersticAlgebraMPB}, we may simply apply Equations \ref{eq:3CharactersticAlgebraMPB}, \ref{eq:4CharactersticAlgebraMPB}, and \ref{eq:5CharactersticAlgebraMPB} in succession. 
\end{proof}

Finally, we note that for mean proximity balanced functions, the First Main Theorem provides an upper bound on how often $f$ can attain $a$. 

\begin{prop}
    Let $f$ be semiregular and mean proximity balanced on $\Omega_D$. Then, \[N(f,a,r)\leq T(f,r)+H(f,a,r)+O(1).\]
\end{prop}
\begin{proof}
    This follows from the fact that the proximity function $m(f,a,r)$ is always nonnegative. 
\end{proof}

\begin{remark}
We note that $H(f,a,r)$ is always nonnegative. Indeed, by Definition~\ref{def:HarmonicRemainderFunction}, $H(f,a,r)$ is defined in terms of the Laplacian of $\log|(f-a)^s|$. \cite[Proposition 8.4]{Bisi2021} proved that $\log|f|^2$ satisfies the mean-value inequality in $\mathbb{R}^4$, and is therefore subharmonic on $\Omega_D \setminus \mathcal{ZP}(f)$. Since $\log|f| = \frac{1}{2}\log|f|^2$, it follows that $\log|f|$ is subharmonic outside the zeroes and poles of $f$. Consequently, for any semiregular function $f$, we have $H(f,a,r) \ge 0$.
\end{remark}

\section{Open Questions}\label{sec:OpenQuestions}

We conclude with several questions suggested by the results of this paper. 

\begin{enumerate}
    \item Is there a precise relationship between the spherical averages of $\log|f|$ and $\log|f\circ S_f|$ for arbitrary semiregular $f$? More generally, given an arbitrary quaternionic function $u(q)$, can the growth of $f(u(q)^{-1}qu(q))$ be controlled solely in terms of the growth of $f(q)$?

    \item Can one formulate an analogue of Jensen's formula, and a consequent notion of value distribution, in which the underlying measure is taken over spherical sets rather than individual quaternionic points? In particular, is it possible to treat each sphere $x+y\s$ as a single atomic element of the measure space?

    \item Is the error term $O(m(ff^c,\infty,r))$ appearing in Theorem~\ref{teo:FMT} best possible for general semiregular $f$, or can it be improved to yield a stronger form of the First Main Theorem in full generality?

    \item Is a Poisson--Jensen formula for mean proximity balanced functions achievable? One possible approach is via Almansi decompositions. A similar approach was utilized by \cite{Perotti2020}. 

    \item Can a Second Main Theorem be established for mean proximity balanced functions?
\end{enumerate}

\bibliographystyle{amsalpha}
\bibliography{bib}
\input{appendix.tex}
\end{document}

%% file: appendix.tex
\appendix
\section{Numerical Verification of the Jensen Formula}

To illustrate the effect of the correction made to Theorem \ref{PJensen} in section \ref{PJSection}, we performed a numerical check for the simple case $f(q)=q-a$ with $a=0.5+0.7i$ and $R=2$. We utilize numerical integration over $\partial \mathbb B_R$. 

\begin{lstlisting}[language=Python, caption={Numerical verification of the Jensen Formula}, label={JensenNumerical}]
import numpy as np

def quat_mul(p, q):
    # Hamilton product
    p0, p1, p2, p3 = p
    q0, q1, q2, q3 = q
    return np.array([
        p0*q0 - p1*q1 - p2*q2 - p3*q3,
        p0*q1 + p1*q0 + p2*q3 - p3*q2,
        p0*q2 - p1*q3 + p2*q0 + p3*q1,
        p0*q3 + p1*q2 - p2*q1 + p3*q0
    ])

def quat_conj(q):
    qc = q.copy()
    qc[1:] *= -1
    return qc

def quat_norm(q):
    return np.linalg.norm(q)

def quat_inv(q):
    n2 = np.dot(q,q)
    return quat_conj(q)/n2

def f(q, a):
    return q - a

def S_f(q, a): #For f(q,a) as defined, the spherical derivative is 1
    fq = f(q, a)
    return quat_mul(quat_mul(quat_inv(fq), quat_conj(q)), fq)

def J_kernel(zeta, R):
    norm_z = quat_norm(zeta)
    re_z = zeta[0]  # real part
    term1 = np.log(R / norm_z)
    term2 = (norm_z**4 - R**4) / (4 * R**2 * norm_z**4) * (2 * (re_z**2) - norm_z**2)
    return term1 + term2

def sample_sphere_3(R, N):
    # sample standard normal and normalize
    x = np.random.normal(size=(N,4))
    norms = np.linalg.norm(x, axis=1)
    x = (R * x.T / norms).T
    return x

# Parameters
R = 2.0
a = np.array([0.5, 0.7, 0.0, 0.0])  # quaternion a = 0.5 + 0.7 i
N = 300000  # samples for Monte Carlo
samples = sample_sphere_3(R, N)

# Compute mean(log|f(w)|) and mean(log|f(S_f(w))|)
logs_f = np.empty(N)
logs_fS = np.empty(N)

for i, w in enumerate(samples):
    fw = f(w, a)
    logs_f[i] = np.log(quat_norm(fw))
    Sf_w = S_f(w, a)
    fS = f(Sf_w, a)
    logs_fS[i] = np.log(quat_norm(fS))

mean_log_f = logs_f.mean()
mean_log_fS = logs_fS.mean()

boundary_term = 0.5 * (mean_log_f + mean_log_fS)

# Derivative terms for f(q)=q-a at q=0
f0 = -a
inv_f0 = quat_inv(f0)

first_derivative = np.array([1.0, 0.0, 0.0, 0.0])
second_derivative = np.array([0.0, 0.0, 0.0, 0.0])

tmp = quat_mul(inv_f0, quat_conj(first_derivative))
tmp_sq = quat_mul(tmp, tmp)
first_term_contrib = - (R**2)/4 * tmp_sq[0]
second_term_contrib = (R**2)/4 * np.real(np.dot(inv_f0, second_derivative))
total_harmonic_contrib = first_term_contrib + second_term_contrib

# Final computation of LHS and RHS
lhs = np.log(quat_norm(a))
rhs_perotti = boundary_term + total_harmonic_contrib - 2 * J_kernel(a, R)
rhs_corrected = boundary_term + total_harmonic_contrib - 1 * J_kernel(a, R)

# Print results
print(f"LHS log|f(0)| = log|a| = {lhs:.12f}")
print(f"Boundary mean log|f| = {mean_log_f:.12f}")
print(f"Boundary mean log|f o S_f| = {mean_log_fS:.12f}")
print(f"Boundary term (average of both /2) = {boundary_term:.12f}")
print(f"Harmonic correction term = {total_harmonic_contrib:.12f}")
print(f"J(a,R) kernel = {J_kernel(a, R):.12f}")
print()
print(f"RHS Perotti (factor 2): {rhs_perotti:.12f}")
print(f"RHS Corrected (factor 1): {rhs_corrected:.12f}")
print()
print("Differences from LHS:")
print(f"Perotti - LHS = {rhs_perotti - lhs:.12e}")
print(f"Corrected - LHS = {rhs_corrected - lhs:.12e}")
\end{lstlisting}

\begin{lstlisting}[caption={Readout of Listing \ref{JensenNumerical}}]
LHS log|f(0)| = log|a| = -0.150552546392
Boundary mean log|f| = 0.739728385939
Boundary mean log|f o S_f| = 0.616708942342
Boundary term (average of both /2) = 0.678218664141
Harmonic correction term = 0.438276113952
J(a,R) kernel = 1.266975840904

RHS Perotti (factor 2): -1.417456903715
RHS Corrected (factor 1): -0.150481062811

Differences from LHS:
Perotti - LHS = -1.266904357323e+00
Corrected - LHS = 7.148358062217e-05
\end{lstlisting}

The above script represents quaternions as 4-vectors and imposes quaternion arithmetic. The integrals are approximate via Monte-Carlo sampling on the 3-sphere with $N=300,000$ samples. In the case where one applies one spherical Blaschke factor per the conventions in Theorem \ref{Jensen}, we attain a difference of $\approx 7.14\times10^{-5}$, which is a small error within numerical precision. On the other hand, in the conventions of Theorem \ref{PJensen}, one attains a difference of $\approx -1.27$, indicating a failure of the identity due to a lack of biharmonicity. 

%% file: bib.bib
@article{Picard1880,
  author  = {Picard, Émile},
  title   = {Mémoire sur les fonctions entières},
  journal = {Annales Scientifiques de l'École Normale Supérieure},
  volume  = {9},
  pages   = {145--166},
  year    = {1880}
}

@article{Nevanlinna1925,
  author  = {Nevanlinna, Rolf},
  title   = {Zur Theorie der meromorphen Funktionen},
  journal = {Acta Mathematica},
  volume  = {46},
  pages   = {1--99},
  year    = {1925}
}

@article{Bernard84,
 ISSN = {00029327, 10806377},
 URL = {http://www.jstor.org/stable/2374283},
 author = {Bernard Shiffman},
 journal = {American Journal of Mathematics},
 number = {3},
 pages = {509--531},
 publisher = {Johns Hopkins University Press},
 title = {A General Second Main Theorem for Meromorphic Functions on Cn},
 urldate = {2025-09-08},
 volume = {106},
 year = {1984}
}

@article{Griffiths1973,
  author  = {Griffiths, Phillip and King, James},
  title   = {Nevanlinna theory and holomorphic mappings between algebraic varieties},
  journal = {Acta Mathematica},
  volume  = {130},
  number  = {1},
  pages   = {145--220},
  year    = {1973},
  doi     = {10.1007/BF02392265},
  url     = {https://doi.org/10.1007/BF02392265},
  issn    = {1871-2509}
}

@book{Noguchi2014,
  author    = {Junjiro Noguchi and Jörg Winkelmann},
  title     = {Nevanlinna Theory in Several Complex Variables and Diophantine Approximation},
  series    = {Grundlehren der mathematischen Wissenschaften},
  edition   = {1},
  publisher = {Springer Tokyo},
  year      = {2014},
  pages     = {XIV, 416},
  isbn      = {978-4-431-54571-2},
  doi       = {10.1007/978-4-431-54571-2},
  url       = {https://doi.org/10.1007/978-4-431-54571-2},
  issn     = {0072-7830},
  eissn    = {2196-9701}
}

@article{Quang2022,
   title={Meromorphic Mappings into Projective Varieties with Arbitrary Families of Moving Hypersurfaces},
   volume={32},
   ISSN={1559-002X},
   url={http://dx.doi.org/10.1007/s12220-021-00765-3},
   DOI={10.1007/s12220-021-00765-3},
   number={2},
   journal={The Journal of Geometric Analysis},
   publisher={Springer Science and Business Media LLC},
   author={Si, Duc Quang},
   year={2022},
   month={jan}}

@article{sudbery1979,
  author    = {Anthony Sudbery},
  title     = {Quaternionic analysis},
  journal   = {Mathematical Proceedings of the Cambridge Philosophical Society},
  volume    = {85},
  number    = {2},
  pages     = {199--224},
  year      = {1979},
  publisher = {Cambridge University Press}
}

@article{GENTILI2007279,
title = {A new theory of regular functions of a quaternionic variable},
journal = {Advances in Mathematics},
volume = {216},
number = {1},
pages = {279-301},
year = {2007},
issn = {0001-8708},
doi = {https://doi.org/10.1016/j.aim.2007.05.010},
url = {https://www.sciencedirect.com/science/article/pii/S0001870807001600},
author = {Graziano Gentili and Daniele C. Struppa},
}

@article{fueter35,
  author  = {Rudolf Fueter},
  title   = {{{\"U}}ber die analytische Darstellung der regul{\"a}ren Funktionen einer Quaternionenvariablen},
  journal = {Commentarii Mathematici Helvetici},
  volume  = {8},
  number  = {1},
  pages   = {371--378},
  year    = {1935},
}

@article{fueter34,
  author  = {Rudolf Fueter},
  title   = {Die Funktionentheorie der Differentialgleichungen $\Delta u = 0$ und $\Delta \Delta u = 0$ mit vier reellen Variablen},
  journal = {Commentarii Mathematici Helvetici},
  volume  = {7},
  number  = {1},
  pages   = {307--330},
  year    = {1934},
}

@book{gurlebeck1990,
  author    = {Klaus G{\"u}rlebeck and Walther Spr{\"o}{\ss}ig},
  title     = {Quaternionic Analysis and Elliptic Boundary Value Problems},
  series    = {International Series of Numerical Mathematics},
  volume    = {89},
  publisher = {Birkh{\"a}user},
  address   = {Basel},
  year      = {1990}
}

@book{kravchenko1996,
  author    = {Vladimir V. Kravchenko and Michael V. Shapiro},
  title     = {Integral Representations for Spatial Models of Mathematical Physics},
  series    = {Pitman Research Notes in Mathematics Series},
  volume    = {351},
  publisher = {Longman},
  address   = {Harlow},
  year      = {1996}
}

@book{gentili2022,
  author    = {Graziano Gentili and Caterina Stoppato and Daniele C. Struppa},
  title     = {Regular Functions of a Quaternionic Variable},
  series    = {Springer Monographs in Mathematics},
  edition   = {2},
  publisher = {Springer Cham},
  year      = {2022},
  month     = {9},
  day       = {25},
  pages     = {XXV, 285},
  isbn      = {978-3-031-07530-8},
  doi       = {10.1007/978-3-031-07531-5},
}

@book{cherry2001nevanlinna,
  author    = {William Cherry and Zhuan Ye},
  title     = {Nevanlinna's Theory of Value Distribution: The Second Main Theorem and its Error Terms},
  series    = {Springer Monographs in Mathematics},
  edition   = {1},
  publisher = {Springer Berlin, Heidelberg},
  year      = {2001},
  isbn      = {978-3-540-66416-1},
  doi       = {10.1007/978-3-662-12590-8},
  pages     = {XII, 203},
}

@article{Ghiloni_2011,
   title={Slice regular functions on real alternative algebras},
   volume={226},
   ISSN={0001-8708},
   url={http://dx.doi.org/10.1016/j.aim.2010.08.015},
   DOI={10.1016/j.aim.2010.08.015},
   number={2},
   journal={Advances in Mathematics},
   publisher={Elsevier BV},
   author={Riccardo Ghiloni and Alessandro Perotti},
   year={2011},
   month=jan, pages={1662–1691} }

@inbook{Perotti_2019,
title={Slice Regularity and Harmonicity on Clifford Algebras},
ISBN={9783030238544},
ISSN={2297-024X},
url={http://dx.doi.org/10.1007/978-3-030-23854-4_3},
DOI={10.1007/978-3-030-23854-4_3},
booktitle={Topics in Clifford Analysis},
publisher={Springer International Publishing},
author={Alessandro Perotti},
year={2019},
pages={53–73} }

@misc{perottiJensen2019,
      title={A four dimensional Jensen formula}, 
      author={Alessandro Perotti},
      year={2019},
      eprint={1902.06485},
      archivePrefix={arXiv},
      primaryClass={math.CV},
      url={https://arxiv.org/abs/1902.06485}, 
}

@article{Altavilla_2019,
title={Log-biharmonicity and a Jensen formula in the space of quaternions},
volume={44},
ISSN={1798-2383},
url={http://dx.doi.org/10.5186/aasfm.2019.4447},
DOI={10.5186/aasfm.2019.4447},
number={2},
journal={Annales Academiae Scientiarum Fennicae Mathematica},
publisher={Finnish Mathematical Society},
author={Altavilla, Amedeo and Bisi, Cinzia},
year={2019},
month=jun, pages={805–839} }

@article{Stoppato_2012,
   title={Singularities of slice regular functions},
   volume={285},
   ISSN={1522-2616},
   url={http://dx.doi.org/10.1002/mana.201100082},
   DOI={10.1002/mana.201100082},
   number={10},
   journal={Mathematische Nachrichten},
   publisher={Wiley},
   author={Stoppato, Caterina},
   year={2012},
   month=feb, pages={1274–1293} }

@book{mitrea2013,
  title={Distributions, Partial Differential Equations, and Harmonic Analysis},
  author={Dorina Mitrea},
  isbn={9781461482086},
  lccn={2013944653},
  series={Universitext},
  url={https://books.google.com/books?id=2ZC4BAAAQBAJ},
  year={2013},
  publisher={Springer New York}
}

@article{Ovall2016,
 ISSN = {00029890, 19300972},
 URL = {https://www.jstor.org/stable/10.4169/amer.math.monthly.123.3.287},
 author = {Jeffrey S. Ovall},
 journal = {The American Mathematical Monthly},
 number = {3},
 pages = {pp. 287--291},
 publisher = {[Taylor & Francis, Ltd., Mathematical Association of America]},
 title = {The Laplacian and Mean and Extreme Values},
 volume = {123},
 year = {2016}
}

@article{Bisi2021,
  author    = {Cinzia Bisi and J{\"o}rg Winkelmann},
  title     = {The Harmonicity of Slice Regular Functions},
  journal   = {The Journal of Geometric Analysis},
  year      = {2021},
  volume    = {31},
  number    = {8},
  pages     = {7773--7811},
  doi       = {10.1007/s12220-020-00551-7},
  url       = {https://doi.org/10.1007/s12220-020-00551-7},
  issn      = {1559-002X}
}

@article{Perotti2020,
  author  = {Alessandro Perotti},
  title   = {Almansi Theorem and Mean Value Formula for Quaternionic Slice-regular Functions},
  journal = {Advances in Applied Clifford Algebras},
  year    = {2020},
  volume  = {30},
  number  = {4},
  pages   = {61},
  issn    = {1661-4909},
  doi     = {10.1007/s00006-020-01078-4}
}
